\documentclass[a4paper,10pt]{article}
\usepackage{fullpage}
\usepackage{amsmath,amssymb,latexsym,amsthm}
\usepackage[pdftex]{graphicx}
\usepackage{graphics}
\usepackage{psfrag}
\usepackage{color, epsfig}
\usepackage{mathrsfs}
\usepackage[frenchb,english]{babel}
\usepackage[latin1]{inputenc}
\usepackage[T1]{fontenc}
\def\R{{\mathbb R}}

\def\S{{\mathcal S}}
\def\T{{\mathbb{T}}}
\def\ds{\displaystyle}

\def\div{\operatorname{div}}
\newcommand{\norm}[1]{\left\Vert#1\right\Vert}
\newtheorem{theorem}{Theorem}[section]

\newtheorem{proposition}[theorem]{Proposition}
\newtheorem{lemma}[theorem]{Lemma}

\numberwithin{equation}{section}

\title{Local exact controllability for the $2$ and $3$-d compressible Navier-Stokes equations\footnote{The authors were partially supported by the ANR-project CISIFS 09-BLAN- 0213-03 and by the project ECOSEA from FNRAE.}}
\author{
Sylvain Ervedoza\footnote{Institut de Mathématiques de Toulouse ; UMR5219; Université de Toulouse ; CNRS ; UPS IMT, F-31062 Toulouse Cedex 9, France, {\tt ervedoza@math.univ-toulouse.fr}}
\and
Olivier Glass\footnote{Ceremade (UMR CNRS no. 7534),
Universit\'e Paris-Dauphine,
Place du Mar\'echal de Lattre de Tassigny,
75775 Paris Cedex 16, France, {\tt glass@ceremade.dauphine.fr}
} 
\and 
Sergio Guerrero\footnote{Universit\'{e} Pierre et Marie Curie - Paris 6,
UMR 7598 Laboratoire Jacques-Louis Lions, Paris, F-75005 France, {\tt guerrero@ann.jussieu.fr}}
}
\date\today
\begin{document}
\maketitle
\begin{abstract}
The goal of this article is to present a local exact controllability result for the $2$ and $3$-dimensional compressible Navier-Stokes equations on a constant target trajectory when the controls act on the whole boundary. Our study is then based on the observability of the adjoint system of some linearized version of the system, which is analyzed thanks to a subsystem for which the coupling terms are somewhat weaker. In this step, we strongly use Carleman estimates in negative Sobolev spaces.
\end{abstract}
%
%
%
%
%
%
%%%%%%%%%%%%%%%%%%%%%%%%%%%%%%%%%%%%%%%%%%%%%%%%%%%%%%%%%%%%%%%%%%%%%%%%%%%%%%%%%%%%%%%%%%%%%%%%%%%%%%%%%%%%%%%%%%%%%%%%%%%%%%%%%%%%%%%
%
%
%
%
\section{Introduction}
%
%
%%%%%%%%%%%%%%%%%%%%%%%%%%%%%%%%%%%%%%%%%%%
%
%
%
%
%
%
%
We consider the isentropic compressible Navier-Stokes equation in dimension two or three in space, in a smooth bounded domain $\Omega \subset \mathbb{R}^d$, $d = 2$ or $d = 3$:
\begin{equation} \label{Navier-Stokes}
\left\{ \begin{array}{ll}
\partial_{t} \rho_\S + \div(\rho_\S u_\S) =0 &\text{ in } (0,T) \times \Omega, \\
\rho_\S (\partial_{t} u_\S + u_\S \cdot \nabla u_\S )  - \mu \Delta u_\S - (\lambda + \mu) \nabla \div (u_\S)
+ \nabla p(\rho_\S) =0 &\text{ in } (0,T) \times \Omega.
\end{array} \right.
\end{equation}
Here $\rho_\S$ is the density, $u_\S$ the velocity and $p$ is the pressure, which follows the standard polytropic law:
\begin{equation} \label{Pression}
p(\rho_\S) = \kappa \rho_\S^{\gamma},
\end{equation}
for some $\gamma \geq 1$ and $\kappa >0$. (Actually, our proof will only require $p$ to be $C^3$ locally around the target density.)\par
The parameters $\mu$ and $\lambda$ correspond to constant viscosity parameters and are assumed to satisfy $\mu >0$ and $d \lambda + 2\mu \geq 0$ (the only condition required for our result is $\mu>0$ and $\lambda + 2 \mu >0$). \par
In this work, we intend to consider the local exact controllability around constant trajectories $(\overline{\rho},\overline{u}) \in \R_+^* \times \R^d \setminus \{ 0\} $. Here, the controls do not appear explicitly in \eqref{Navier-Stokes} as we are controlling the whole external boundary $(0,T) \times \partial \Omega$ for the equation of the velocity and the incoming part $ u_\S \cdot \vec{n} <0$ of the boundary for the equation of the density, $\vec{n}$ being the unit outward normal on $\partial \Omega$, see e.g. \cite[Chapter 5]{PlotnikovSololowski}.
\par
Given $e$ a direction in the unit sphere $\mathbb{S}^{d-1}$ of $\R^{d}$, we define the {\it thickness} of some nonempty open set $A \subset \R^{d}$ in the direction $e$ as the following nonnegative number:
\begin{equation*}
\sup \left\{ \ell \geq 0 \ / \ \exists x \in A, \, x + \ell e \in A \right\}.
\end{equation*}
Our main result is the following.
\begin{theorem}\label{Thm-Main}
	Let $d \in \{2, 3\}$, $\overline{\rho} >0$ and $\overline{u} \in \R^d\setminus\{0\}$. Let $L_0>0$ be larger than the thickness of $\Omega$ in the direction $\overline{u}/|\overline{u}|$, and assume 
	\begin{equation}
		\label{Time-Condition}
		T > L_0/|\overline{u}|.
	\end{equation}
	
	There exists $\delta>0$ such that for all $(\rho_0, u_0) \in H^2(\Omega) \times H^2(\Omega)$ satisfying 
	\begin{equation}
		\label{Smallness-Init}
		\norm{(\rho_0, u_0)}_{H^2(\Omega) \times H^2(\Omega)} \leq \delta, 
	\end{equation}
	there exists a solution $(\rho_\S, u_\S)$ of \eqref{Navier-Stokes} with initial data
	\begin{equation}
		\label{NS-Init}
		\rho_\S(0, x) = \overline \rho+ \rho_0(x), \quad u_\S(0,x) = \overline u + u_0(x) \quad \text{ in } \Omega,
	\end{equation}
	and satisfying the control requirement
	\begin{equation}
		\label{NS-Goal}
		\rho_\S(T,x) = \overline \rho, \quad u_\S(T,x) = \overline u \quad \text{ in } \Omega.
	\end{equation}
	Besides, the controlled trajectory $(\rho_\S, u_\S)$ has the following regularity:
	\begin{equation}
		\label{Regularity-Controlled-Data}
		(\rho_\S, u_\S) \in C([0,T]; H^2(\Omega)) \times (L^2(0,T; H^3(\Omega)) \cap C^0([0,T]; H^2(\Omega))).
	\end{equation}
\end{theorem}
\par
Theorem \ref{Thm-Main} extends the results in \cite{EGGP} to the multi-dimensional case. As in the one dimensional case, our result proves local controllability to constant states having non-zero velocity. This restriction appears explicitly in the condition \eqref{Time-Condition}. As expected, this condition is remanent from the transport equation satisfied by the density which allows the information to travel at a velocity (close to) $\overline{u}$. 
\par
This transport phenomenon and its consequences on the controllability of compressible Navier-Stokes equations have been also developed and explained in the articles \cite{RosierRouchon,Chow-Ram-Raymond,Debayan-2014} focusing on the linearized equations in the case of zero-velocity. Using then moving controls, \cite{MartinRosierRouchon,ChavesRosierZuazua-2013} managed to show that controllability for a system of linear viscoelasticity can be reestablished if the control set travels in the whole domain (among some other geometric conditions, see \cite{ChavesRosierZuazua-2013} for further details). 
Let us also mention the work \cite{ChowdhuryDebanjanaRamaswamyRenardy} where the compressible Navier-Stokes equations in $1$d linearized around a constant state with non-zero velocity are studied thoroughly using a spectral approach and suitable Ingham-type inequalities. \par
\ \par
In order to prove Theorem \ref{Thm-Main}, we will deal with system \eqref{Navier-Stokes}--\eqref{Pression} thinking to it as a coupling of parabolic and transport equations, and we shall therefore borrow some ideas from previous works studying controllability of systems coupling parabolic and hyperbolic effects, in particular the works \cite{AlbanoTataru} focusing on a system of linear thermoelasticity, \cite{EGGP} for the $1$d compressible Navier-Stokes equation around a constant state with non-zero velocity, \cite{ChavesRosierZuazua-2013} for a system of viscoelasticity with moving controls, or \cite{BEG} for non-homogeneous incompressible Navier-Stokes equations. All these works are all based on suitable Carleman estimates designed simultaneously for the control of a parabolic equation following the ideas in \cite{FursikovImanuvilov} and for the control of the hyperbolic equation.
\par
Our approach will follow this path and use Carleman estimates with weight functions traveling at velocity $\overline u$ similarly as in \cite{EGGP,ChavesRosierZuazua-2013,BEG}. But we will also need to construct smooth trajectories in order to guarantee that the velocity field belongs to $L^2(0,T;H^3(\Omega))$. This space is natural as it is included in the space $L^1(0,T; \mbox{Lip}(\Omega))$ ensuring the existence and continuity of the flow. Therefore, in order to obtain velocity fields in $L^2(0,T;H^3(\Omega))$, we will use duality and develop observability estimates in negative Sobolev spaces in the spirit of the work \cite{ImaYam}. 
\par
We will not deal with the Cauchy problem for system \eqref{Navier-Stokes}--\eqref{Pression}, as our strategy directly constructs a solution of \eqref{Navier-Stokes}--\eqref{Pression}. We refer the interested reader to the pioneering works by P.-L. Lions \cite{LionsPL-Compressible} and E. Feireisl \emph{et al.} in \cite{Feireisl-Novotny-Petzeltova}. Nevertheless, we emphasize that our approach will use on the adjoint equations a new variable which is similar to the so-called viscous effective flux introduced by P.-L. Lions in \cite{LionsPL-Compressible} in order to gain compactness properties. 
\par
Let us briefly mention other related works in the literature. In particular, we shall quote the works on the controllability of compressible Euler equations, namely the ones obtained in \cite{LiRao} in the $1$-dimensional setting in the context of classical $C^1$ solutions, and the ones developed by the second author in the context of weak entropy solutions obtained in \cite{Glass-EulerComp} for isentropic $1$-d Euler equations and \cite{Glass-2014-Euler} for non-isentropic $1$d Euler equations. We also refer to the work \cite{Nersisyan} for a global approximate controllability result for the $3$-d Euler equations with controls  spanned by a finite number of modes.
When considering incompressible flows, the literature is large. We refer for instance to the works \cite{Imanuvilov2001,FerCarGueImaPuel,ImanuvilovPuelYam-2009} for several results on the local exact controllability to trajectories for the (homogeneous) incompressible Navier-Stokes equations, and to the works \cite{Co1,Glass-3d} for global exact controllability results for incompressible perfect fluids. \par
\ \par
\noindent
{\it Outline.} The article is organized as follows. Section \ref{Sec-Strategy} presents the general strategy of the proof of Theorem \ref{Thm-Main}. Section \ref{Sec-Control-Sous-System} shows the controllability of a suitable system of one parabolic and one transport equation. Section \ref{Sec-Control-NS-Linear} deduces from it a controllability result for the linearized Navier-Stokes equations. Section \ref{Sec-Proof-Main} then explains how to perform a fixed point argument using the controllability results developed beforehand, thus proving Theorem~\ref{Thm-Main}. Section \ref{Sec-Further} provides some open problems.
%%%%%%%%%%%%%%%%%%%%%%%%%%%%%%%%%%%%%%%%%%%%%%%%%%%%%%%%%%%%%%%%%%%%%%%%%%%%%%%%%%%%%%%%%%%%%%%%%%%%%%%%%%%%%%%%%%%%%%%%%%%%%%%%%%%%%%%
%
%
%
%
\section{General strategy}\label{Sec-Strategy}
%
%
%
%%%%%%%%%%%%%%%%%%%%%%%%%%%%%%%%%%%%%%%%%%%%%%
%
%
\subsection{Main steps of the proof}
\label{Subsec:MainSteps}
Since we are controlling the whole external boundary, $\Omega$ can be embedded into some torus $\T_L$, where $\T_{L}$ is identified with $[0,L]^d$ with periodic conditions. The length $L$ is large enough (for instance $L = \hbox{diam}(\Omega) + 5| \overline u| T$) and we may consider the control problem in the cube $[0,L]^d$ completed with periodic boundary conditions with controls appearing as source terms supported in $ \T_L \setminus \overline{\Omega}$.
Our control system then reads as follows:
\begin{equation} \label{Navier-Stokes-TL}
\left\{ \begin{array}{ll}
\partial_{t} \rho_\S + \div(\rho_\S u_\S) = \check v_\rho, & \text{ in } (0,T) \times \T_L, \\
\rho_\S (\partial_{t} u_\S + u_\S \cdot \nabla u_\S )  - \mu \Delta u_\S - (\lambda + \mu) \nabla \div (u_\S) 
+ \nabla p(\rho_\S) =\check v_u,  & \text{ in } (0,T) \times \T_L,
\end{array} \right.
\end{equation}
where $\check v_\rho$ and $\check v_u$ are control functions supported in $[0,T] \times (\T_L \setminus \overline\Omega)$.
Then we set 
\begin{equation*}
	\check \rho: = \rho_\S - \overline{\rho}, \quad \check u := u_\S - \overline{u}.    
\end{equation*}
We also extend the initial data $(\rho_0, u_0)$ to $\T_L$ such that 
\begin{equation}
	\label{Norm-Init-check}
	\norm{(\check \rho_0, \check u_0)}_{H^2(\T_L) \times H^2(\T_L)} \leq C_L \norm{(\rho_0, u_0)}_{H^2(\Omega) \times H^2(\Omega)} \leq C_L \delta.
\end{equation}
With these notations, one needs to solve the following control problem: Given $(\check \rho_0, \check u_0)$ small in $H^2(\T_L) \times H^2(\T_L)$, find control functions $\check v_\rho$ and $\check v_u$ supported in $[0,T]\times (\T_L \setminus \overline\Omega)$ such that the solution of 
\begin{equation} \label{Navier-Stokes-Tothefixpoint-0}
\left\{ \begin{array}{ll}
\ds \partial_{t} \check\rho +  (\overline{u} + \check u)\cdot \nabla \check\rho + \overline{\rho} \div(\check u) = \check v_\rho + \check f_\rho (\check \rho, \check u), & \text{ in } (0,T) \times \T_L, \\
\ds \overline{\rho} (\partial_{t} \check u + (\overline{u} + \check u)\cdot \nabla \check u )  - \mu \Delta \check u - (\lambda + \mu) \nabla \div (\check u) 
+ p'(\overline{\rho}) \nabla \check \rho =\check v_u + \check f_u (\check\rho, \check u), & \text{ in } (0,T) \times \T_L,
\end{array} \right.
\end{equation}
with initial data
\begin{equation} \label{InitData-0}
\check \rho(0, x) = \check \rho_0(x), \quad \check u(0,x) = \check u_0(x) \quad \text{ in } \T_L,
\end{equation}
and source terms
\begin{eqnarray}
	\label{SourceTerm-Tilde-F-rho}
	\check f_\rho(\check \rho, \check u ) &=& - \check \rho \div(\check u), 
	\\
	\label{SourceTerm-Tilde-F-u}
	\check f_u(\check \rho, \check u) &=& - \check \rho (\partial_t \check u + (\overline{u} + \check u) \cdot \nabla \check u) 
- \nabla (p(\overline{\rho} + \check \rho) - p'(\overline{\rho}) \check \rho),
\end{eqnarray}
satisfies
\begin{equation} \label{ControlReq-0}
\check \rho(T) = 0, \quad \check u(T) = 0 \quad \text{ in } \T_L.
\end{equation}
To take the support of the control functions $\check v_\rho$ and $\check v_u$ into account, we introduce a smooth cut-off function $\chi \in C^{\infty}(\T_{L};[0,1])$  satisfying
\begin{equation}
	\label{Def-Chi}
	\left\{
		\begin{array}{l}
		\ds \chi (x) =0 \ \text{ for all } x \text{ such that } d(x, \Omega) \leq \varepsilon, 
		\\
		\ds \chi (x) =1 \ \text{ for all } x \text{ such that } d(x, \Omega) \geq 2 \varepsilon,
		\end{array}
	\right.
\end{equation}
and we will look for $\check v_\rho$ and $\check v_u$ in the form
\begin{equation*}
\check v_\rho = v_{\rho} \chi \ \text{ and } \check v_u = v_{u} \chi.
\end{equation*}
Now in order to solve the controllability problem \eqref{Navier-Stokes-Tothefixpoint-0}--\eqref{ControlReq-0}, we will use a fixed point argument. 
A difficulty arising when building this argument is that the term $\check u \cdot \nabla \check\rho$ in \eqref{Navier-Stokes-Tothefixpoint-0}$_{(1)}$ is very singular.
Hence we start by removing this term via a diffeomorphism close to the identity. To be more precise, we define the flow $X_{\check u} = X_{\check u}(t,\tau, x)$ corresponding to $\check u$ and defined for $(t, \tau, x) \in [0,T] \times [0,T] \times \T_L$ by the equation
\begin{equation}
	\label{Flow-u}
	\frac{dX_{\check u}}{dt}(t,\tau, x) = \overline{u} + \check u(t,X_{\check u}(t,\tau,x)), \quad t \in [0,T], \qquad X_{\check u}(\tau, \tau, x) = x.
\end{equation}
In order to give a sense to \eqref{Flow-u}, we require $\check u\in L^1(0,T; W^{1,1}(\T_L))$ and $\div(\check u) \in L^1(0,T; W^{1, \infty}(\T_L))$ (see \cite{DiPernaLions-89}). But as we will work in Hilbert spaces with integer indexes, we will rather assume the stronger assumption $\check u \in L^2(0,T; H^3(\T_L))$, in which case, the flow $X_{\check u}$ is well-defined classically by Cauchy-Lipschitz's theorem. 
We then set, for $(t,x) \in [0,T] \times \T_L$,
\begin{equation}
	\label{RedressingFlow}
	Y_{\check u}(t,x) = X_{\check u} (t,T, X_0(T,t,x)),  
	\quad
	Z_{\check u}(t,x) = X_0(t,T, X_{\check u} (T,t,x)),
\end{equation}
which are inverse one from another, i.e. $Y_{\check u}(t, Z_{\check u}(t,x)) = Z_{\check u}(t,Y_{ \check u}(t,x)) = x$ for all $(t,x) \in [0,T] \times \T_L$.
For $\check u$ suitably small, both transformations $Y_{\check u}(t,\cdot)$ and $Z_{\check u}(t,\cdot)$, $t \in [0,T]$, are diffeomorphism of $\T_L$ which are close to the identity map on the torus. 
This change of variable is reminiscent of the Lagrangian coordinates and allow to straighten the characteristics. \par 
We thus set, for $(t,x) \in [0,T] \times \T_L$,
\begin{equation}
	\label{Dico-rho-u-tilde-rho-u}
	\rho(t,x) = \check \rho (t,Y_{\check u}(t,x)), \quad u(t,x) = \check u(t, Y_{\check u}(t,x)), 
\end{equation}
After tedious computations developed in Appendix \ref{Appendix}, our problem can now be reduced to find controlled solutions $(\rho, u)$ of
\begin{equation} \label{Navier-Stokes-Tothefixpoint}
\left\{ \begin{array}{ll}
\ds \partial_{t} \rho +  \overline{u} \cdot \nabla \rho + \overline{\rho} \div(u) = v_\rho \chi + f_\rho (\rho, u), & \text{ in } (0,T) \times \T_L, \\
\ds \overline{\rho} (\partial_{t} u + \overline{u} \cdot \nabla u )  - \mu \Delta u - (\lambda + \mu) \nabla \div ( u) 
+ p'(\overline{\rho}) \nabla \rho = v_u \chi + f_u (\rho, u), & \text{ in } (0,T) \times \T_L,
\end{array} \right.
\end{equation}
for some $\varepsilon >0$ small enough, with initial data given by
\begin{equation} \label{InitData-0-1}
	\rho(0, x) = \check \rho_0(Y_{\check u}(0,x)), \quad u(0,x) = \check u_0(Y_{\check u}(0, x)) \quad \text{ in } \T_L,
\end{equation}
and source terms $f_\rho(\rho, u)$ given by 
\begin{equation*}
	f_\rho(\rho, u)  = 
	- \rho D Z_{\check u}^t(t, Y_{\check u} (t,x)):Du 	
	 - \overline{\rho} ( DZ_{\check u}^t (t,Y_{\check u} (t,x)) - I):Du, 
\end{equation*}
and $f_u(\rho, u)$ by 
\begin{align*}
	& 
	f_{i,u}(\rho, u) 
	= 
	- \rho (\partial_t u_i + \overline u \cdot \nabla u_i) 
	+ 
	\sum_{j = 1}^d \partial_i Z_{j,\check u}(t,Y_{\check u}(t,x)) \partial_j ( p(\overline \rho + \rho) - p'(\overline\rho) \rho)
	\\
	&
	+ \mu 
	\left(
	 \sum_{j,k, \ell = 1}^d \partial_{k,\ell} u_i \left(\partial_j Z_{k, \check  u}(t, Y_{\check  u}(t,x)) - \delta_{j,k}\right)\left(\partial_j Z_{\ell, \check  u}(t, Y_{\check  u}(t,x)) - \delta_{j,\ell}\right) + \sum_{k = 1}^d \partial_k u_i \Delta Z_{k, \check  u}(t,Y_{\check  u}(t,x)) 
	\right)
	\\
	&
	+ (\lambda + \mu) 
	\left( 	
	\sum_{j,k, \ell =1}^d \partial_{k, \ell} u_j (\partial_j Z_{k,\check  u} (t,Y_{\check  u}(t,x)) - \delta_{j,k}) (\partial_i Z_{\ell, \check  u}(t,Y_{\check  u}(t,x)) - \delta_{i, \ell})
	\right) 
	\\
	& 
	+(\lambda +\mu) 
	\left(
		\sum_{j,k= 1}^d \partial_{i,j} Z_{k, \check u}(t, Y_{\check  u}(t,x)) \partial_k u_j
	\right) 
%	\\
%	& 
	- p'(\overline{\rho}) 
	\left(
	 \sum_{j = 1}^d ( \partial_i Z_{j,\check u}(t,Y_{\check u}(t,x)) - \delta_{i,j}) \partial_j \rho
	\right),
\end{align*}
where $\delta_{j,k}$ is the Kronecker symbol ($\delta_{j,k} = 1$ if $j = k$, $\delta_{j,k} = 0$ if $j \neq k$), 
and satisfying the controllability requirement
\begin{equation} \label{ControlReq}
	\rho(T) = 0, \quad u(T) = 0 \quad \text{ in } \T_L.
\end{equation}
The corresponding control functions in \eqref{Navier-Stokes-Tothefixpoint} will then be given for $(t,x) \in [0,T] \times \T_L$ by 
\begin{equation}
	 \check v_\rho(t,x ) = \chi (Z_{\check u}(t,x)) v_\rho(t,Z_{\check u}(t,x)), \quad \check v_u(t,x) = \chi(Z_{\check u}(t,x)) v_u(t,Z_{\check u}(t,x)), 
\end{equation}
which are supported in $[0,T] \times (\T_L \setminus \overline\Omega)$ provided that
\begin{equation}
	\label{Condition-Chi}
	\chi(Z_{\check u}(t,x)) = 0 \hbox{ for all } (t,x) \in [0,T] \times \overline{\Omega}.
\end{equation}
Let us then remark that the map $Y_{\check u}$ can be computed starting from $u$. Indeed, we have that $Y_{\check u}(t,X_0(t, T,x)) = X_{\check u}(t,T,x)$ so that by differentiation with respect to the time variable, we get, for all $(t,x) \in [0,T] \times \T_L$,
$$
	\partial_t Y_{\check u}(t, X_0(t,T,x)) + \overline u \cdot \nabla Y_{\check u}(t,X_0(t,T,x)) = \overline u + \check u(t,X_{\check u} (t,T,x)). 
$$
In particular, using this equation at the point $X_0(T,t,x)$, we obtain
$$
	\partial_t Y_{\check u} (t,x) + \overline u \cdot \nabla Y_{\check u}(t,x) 
	= 
	\overline u + \check u(t,X_{\check u}(t,T, X_0(T,t,x))) 
	= 
	\overline u + \check u(t, Y_{\tilde u}(t,x))
	= 
	\overline u + u(t,x).
$$
\par
Next we shall introduce a map $\mathscr{F}: (\widehat {\rho}, \widehat {u}) \mapsto (\rho,u)$ defined on a convex subset of some weighted Sobolev spaces, corresponding to some Carleman estimate described later.
This fixed point map is constructed as follows. Given $(\widehat {\rho}, \widehat {u})$ small in a suitable norm, we first define $\widehat Y= \widehat Y(t,x)$ as the solution of 
\begin{equation}
	\label{Def-Hat-Y}
	\partial_t \widehat Y + \overline u \cdot \nabla \widehat Y = \overline u + \widehat  u, \hbox{ in }(0,T) \times \T_L, 
	\qquad \widehat Y(T, x) = x, \hbox{ in } \T_L,
\end{equation}
Then we define $\widehat Z = \widehat Z(t,x)$ as follows: for all $t \in [0,T]$, $\widehat Z(t, \cdot)$ is the inverse of $\widehat Y(t,\cdot)$ on $\T_L$. In other words, for all $(t,x ) \in [0,T] \times \T_L$,
\begin{equation}
	\label{Def-Hat-Z}
	\widehat Z(t,\widehat Y(t,x)) = x, \quad \widehat Y(t,\widehat Z(t,x)) = x. 
\end{equation}
We will see that for suitably small $\widehat {u}$, $\widehat Y(t,\cdot)$ is invertible for all $t \in [0,T]$, see Proposition \ref{Prop-Hat-Y}. \par
Corresponding to the initial data, we introduce
\begin{equation}
	\label{Dico-Tilde-rho-u-0-rho-u}
	\widehat  \rho_0(x) = \check \rho_0(\widehat Y(0,x)), \quad \widehat  u_0(x) = \check u_0(\widehat Y(0,x)), \hbox{ in } \T_L,
\end{equation}
and, corresponding to the source terms,
\begin{equation}
\label{SourceTermRho-hat}
	f_\rho(\widehat \rho, \widehat  u)  = 
	- \widehat \rho D \widehat Z^t(t, \widehat Y(t,x)):D\widehat u 	
	 - \overline{\rho} ( D\widehat Z^t (t,\widehat Y(t,x)) - I):D\widehat u, 
\end{equation}
and
\begin{align}
\label{SourceTermU-hat}	
& 
	f_{i,u}(\widehat  \rho, \widehat  u) 
	= 
	- \widehat \rho (\partial_t \widehat u_i + \overline u \cdot \nabla \widehat u_i) 
	+ 
	\sum_{j = 1}^d \partial_i \widehat Z_{j}(t,\widehat Y(t,x)) \partial_j ( p(\overline \rho + \widehat \rho) - p'(\overline\rho) \widehat \rho).
	\\
	&
	+ \mu 
	\left(
	 \sum_{j,k, \ell = 1}^d \partial_{k,\ell} \widehat u_i \left(\partial_j \widehat Z_{k}(t, \widehat Y(t,x)) - \delta_{j,k}\right)\left(\partial_j \widehat Z_{\ell}(t, \widehat Y(t,x)) - \delta_{j,\ell}\right) + \sum_{k = 1}^d \partial_k \widehat u_i \Delta \widehat Z_{k}(t,\widehat Y(t,x)) 
	\right)
	\notag
	\\
	&
	+ (\lambda + \mu) 
	\left( 	
	\sum_{j,k, \ell =1}^d \partial_{k, \ell} \widehat u_j (\partial_j \widehat Z_{k} (t,\widehat Y(t,x)) - \delta_{j,k}) (\partial_i \widehat Z_{\ell}(t,\widehat Y(t,x)) - \delta_{i, \ell})
	\right) 
	\notag
	\\
	& 
	+(\lambda +\mu) 
	\left(
		\sum_{j,k= 1}^d \partial_{i,j} \widehat Z_k(t, \widehat Y(t,x)) \partial_k \widehat u_j
	\right) 
	- p'(\overline{\rho}) 
	\left(
	 \sum_{j = 1}^d ( \partial_i \widehat Z_{j}(t,\widehat Y(t,x)) - \delta_{i,j}) \partial_j \widehat \rho
	\right).
	\notag
\end{align}
\par
We then look for $(\rho, u)$ solving the controllability problem
\begin{equation} \label{Navier-Stokes-fixpoint}
\left\{ \begin{array}{ll}
\ds \partial_{t} \rho +  \overline{u}\cdot \nabla \rho + \overline{\rho} \div(u) = v_\rho \chi + f_\rho (\widehat {\rho}, \widehat {u}), 
& \text{ in } (0,T) \times \T_L, \\
\ds \overline{\rho} (\partial_{t} u + \overline{u} \cdot \nabla u ) - \mu \Delta u - (\lambda + \mu) \nabla \div (u)+ p'(\overline{\rho}) \nabla \rho 
=v_u \chi + f_u (\widehat {\rho}, \widehat {u}), & \text{ in } (0,T) \times \T_L,
\end{array} \right.
\end{equation}
with initial data 
\begin{equation} \label{InitData}
	\rho(0, x) = \widehat  \rho_0(x) , \quad u(0,x) = \widehat  u_0(x) \quad \text{ in } \T_L,
\end{equation}
with source terms $f_\rho(\widehat  \rho, \widehat  u), f_u(\widehat  \rho, \widehat  u)$ as in \eqref{SourceTermRho-hat}--\eqref{SourceTermU-hat}, 
and satisfying the controllability objective \eqref{ControlReq}.
 \par
We are therefore reduced to study the controllability of the linear system
\begin{equation} \label{Navier-Stokes-Linear}
\left\{ \begin{array}{ll}
\ds \partial_{t} \rho +  \overline{u} \cdot \nabla \rho + \overline{\rho} \div(u) 
= v_\rho \chi + \widehat {f}_\rho, & \text{ in } (0,T) \times \T_L, \\
\ds \overline{\rho} (\partial_{t} u + \overline{u} \cdot \nabla u ) - \mu \Delta u - (\lambda + \mu) \nabla \div (u)
+ p'(\overline{\rho}) \nabla \rho =v_u \chi +\widehat { f}_u, & \text{ in } (0,T) \times \T_L.
\end{array} \right.
\end{equation}
Since this is a linear system, the controllability of \eqref{Navier-Stokes-Linear} is equivalent to the observability property for the adjoint equation 
\begin{equation} \label{Navier-Stokes-Adjoint}
\left\{ \begin{array}{ll}
\ds - \partial_{t} \sigma -  \overline{u} \cdot \nabla \sigma - p'(\overline{\rho}) \div(z) = g_\sigma, &\text{ in } (0,T) \times \T_L,  \\
\ds - \overline{\rho} (\partial_{t} z + \overline{u} \cdot \nabla z )  - \mu \Delta z - (\lambda + \mu) \nabla \div (z) 
- \overline{\rho}  \nabla \sigma =g_z, &\text{ in } (0,T) \times \T_L.
\end{array} \right.
\end{equation}
The main idea to get an observability inequality for \eqref{Navier-Stokes-Adjoint} is to remark that, taking the divergence of the equation of $z$, the equations of $\sigma$ and $\div(z)$ form a closed coupled system:
\begin{equation} \label{Navier-Stokes-Adjoint-Div-Sigma}
\left\{ \begin{array}{ll}
\ds - \partial_{t} \sigma - \overline{u}\cdot \nabla \sigma - p'(\overline{\rho}) \div(z) = g_\sigma, &\text{ in } (0,T) \times \T_L, \smallskip \\
\ds - \overline{\rho} (\partial_{t} \div(z) + \overline{u} \cdot \nabla \div(z) ) - \nu \Delta \div(z)  - \overline{\rho}  \Delta \sigma
= \div (g_z), &\text{ in } (0,T) \times \T_L, 
\end{array} \right.
\end{equation}
where 
\begin{equation*}
\nu := \lambda + 2 \mu > 0.
\end{equation*}
Now we are led to introduce a new variable $q$ as follows:
\begin{equation} \label{Def-q}
q := \nu \div(z) + \overline{\rho} \sigma.
\end{equation}
System \eqref{Navier-Stokes-Adjoint-Div-Sigma} can then be rewritten with the unknown $(\sigma, q)$ as
\begin{equation} \label{Navier-Stokes-Adjoint-Mu-Sigma}
\left\{ \begin{array}{ll}
\ds - \partial_{t} \sigma - \overline{u} \cdot \nabla \sigma + \frac{p'(\overline{\rho})\overline{\rho} }{\nu} \sigma 
=   g_\sigma+ \frac{p'(\overline{\rho})}{\nu} q , & \text{ in } (0,T) \times \T_L, \smallskip \\
\ds - \frac{\overline{\rho}}{\nu} (\partial_{t} q + \overline{u} \cdot \nabla q) -  \Delta q -  \frac{p'(\overline{\rho})\overline{\rho}^2}{\nu^2} q
= \div (g_z) + \frac{\overline{\rho}^2}{\nu} g_\sigma 
- \frac{p'(\overline{\rho})\overline{\rho}^3}{\nu^2} \sigma 	,& \text{ in } (0,T) \times \T_L.
\end{array} \right.
\end{equation}
The advantage of considering system \eqref{Navier-Stokes-Adjoint-Mu-Sigma} rather than \eqref{Navier-Stokes-Adjoint-Div-Sigma} directly is that now the coupling between the two equations is of lower order.
In particular, the observability can now be obtained directly by considering independently the observability inequality for the equation of $\sigma$, which is of transport type, and for the equation of $z$, which is of parabolic type, considering in both cases the coupling term as a source term. %\par
Let us emphasize that the quantity $q$ in \eqref{Def-q} can be seen as a version of the so-called effective viscous flux $\nu \div (u) - p(\rho)$, which has been used for the analysis of the Cauchy problem for compressible fluids \cite{Feireisl-Novotny-Petzeltova,LionsPL-Compressible}, but for the dual operator. \par
Now, let us again remark that as system \eqref{Navier-Stokes-Adjoint-Mu-Sigma} is linear, its observability is equivalent to a controllability statement for the adjoint equation written in the dual variables $(r, y)$, where the adjoint is taken with respect to the variables $(\sigma, q)$. This leads to the controllability problem:
\begin{equation} \label{Control-ReducedEq}
\left\{ \begin{array}{ll}
\ds  \partial_{t} r + \overline{u}  \cdot \nabla r + \frac{p'(\overline{\rho})\overline{\rho} }{\nu} r
=  f_r   - \frac{p'(\overline{\rho})\overline{\rho}^3}{\nu^2} y + v_r \chi_0, 
& \text{ in } (0,T) \times \T_L,  \smallskip\\
\ds  \frac{\overline{\rho}}{\nu} (\partial_{t} y + \overline{u} \cdot \nabla y) -  \Delta y -  \frac{p'(\overline{\rho})\overline{\rho}^2}{\nu^2} y
= f_y + \frac{p'(\overline{\rho})}{\nu} r + v_y \chi_0,  & \text{ in } (0,T) \times \T_L, \\
(r(0, \cdot), y(0,\cdot)) = (r_0, y_0) & \text{ in } \T_L, \\
(r(T, \cdot), y(T,\cdot)) = (0,0) & \text{ in } \T_L, 
\end{array} \right.
\end{equation}
where in order to add a margin on the control zone we introduce $\chi_0$ is a smooth cut-off function satisfying 
\begin{equation}
	\label{Def-Chi0}
	\hbox{Supp}\, \chi_0 \Subset \{ \chi = 1\} \hbox{ and } \chi_0(x) = 1 \hbox{ for all }x \in \T_L \hbox{ such that } d(x, \Omega) \geq 3 \varepsilon.
\end{equation}
Now in order to solve the controllability problem \eqref{Control-ReducedEq}, we use again another fixed point argument, and begin by considering the following decoupled controllability problem:
\begin{equation} \label{Control-Hyper-ReducedEq}
\left\{ \begin{array}{ll}
\ds  \partial_{t} r + \overline{u} \cdot \nabla r + \frac{p'(\overline{\rho})\overline{\rho} }{\nu} r =   \tilde f_r + v_r \chi_0, 
& \text{ in } (0,T) \times \T_L, \smallskip\\ 
\ds  \frac{\overline{\rho}}{\nu} \partial_{t} y  -  \Delta y 
= \tilde f_y + v_y \chi_0,  & \text{ in } (0,T) \times \T_L, \\
(r(0, \cdot), y(0,\cdot)) = (r_0, y_0) & \text{ in } \T_L, \\
(r(T, \cdot), y(T,\cdot)) = (0,0) & \text{ in } \T_L.
\end{array} \right.
\end{equation}
Getting suitable estimates on the controllability problem \eqref{Control-Hyper-ReducedEq} will allow us to solve the controllability problem \eqref{Control-ReducedEq} by a fixed point argument. %\par
Note that the control problem for \eqref{Control-Hyper-ReducedEq} simply consists of the control of two decoupled equations, the one in $r$ of transport type, the other one in $y$ of parabolic type. We are then reduced to these two classical problems. \par
It turns out that our main difficulty then will be to show the existence of smooth controlled trajectory for smooth source terms. Indeed, this is needed as we would like to consider velocity fields $u \in L^1(0,T; \mbox{Lip}(\T_L))$. As Carleman estimates are the basic tools to establish the controllability of parabolic equations and to estimate the regularity of controlled trajectories and since they are based on the Hilbert structures of the underlying functional spaces, it is therefore natural to try getting velocity fields 
\begin{equation*}
u \in L^2(0,T; H^3(\T_L)) \cap H^1(0,T; H^1(\T_L)) \cap C^0([0,T]; H^2(\T_L)). 
\end{equation*}
This regularity corresponds to the following regularity properties on the other functions:
\begin{itemize}
	\item $g_z \in L^2(0,T; H^{-3}(\T_L))$, $z \in L^2(0,T; H^{-1}(\T_L))$, 
	\item $q \in L^2(0,T; H^{-2}(\T_L))$, $\sigma \in L^2(0,T; H^{-2}(\T_L))$, $g_\sigma \in L^2(0,T;H^{-2}(\T_L))$,
	\item $f_r, f_y, \tilde f_r, \tilde f_y \in L^2(0,T; H^2(\T_L))$, $r \in L^2(0,T; H^2(\T_L))$, $y \in L^2(0,T;H^4(\T_L))$.
\end{itemize}

%
%
%
%%%%%%%%%%%%%%%%%%%%%%%%%%%%%%%%%%%%%%%%%%%%%%%%%%%%%%%%%%%%%%%%%%%%%%
%
%
%
%
\subsection{Construction of the weight function} %
The controllability and observability properties of the systems described in Section \ref{Subsec:MainSteps} will be studied by using Carleman estimates.
These require the introduction of several notations, in particular to define the weight function.
We first construct a function $ \tilde \psi = \tilde \psi(t,x) \in C^2([0,T] \times \T_L)$ satisfying the following properties. \par
\ \par
\noindent
{\bf 1.} First, it is assumed that
\begin{equation} \label{Psi}
\forall (t,x) \in [0,T] \times \T_L,\ \   \tilde \psi(t,x)  \in [0,1].
\end{equation}
{\bf 2.} We assume that $\tilde \psi$ is constant along the characteristics of the target flow, i.e. $\tilde \psi$ solves
\begin{equation} \label{TransportPoids}
\partial_t \tilde \psi + \overline{u}\cdot \nabla \tilde \psi = 0 \hbox{ in } (0,T) \times \T_L, 
\end{equation}
or equivalently
\begin{equation} \label{TransportPoids2}
\tilde \psi(t,x) = \Psi(x - \overline{u} t) \ \text{ with } \ \Psi(\cdot):=\tilde{\psi}(0,\cdot).
\end{equation}
\noindent
{\bf 3.} We finally assume the existence of a subset $\omega \Subset \{ \chi_0 = 1\}$ such that
\begin{equation} \label{Ass-Psi}
\inf \left\{|\nabla \tilde \psi (t,x)|, \ (t,x) \in [0,T] \times (\T_L \setminus \omega) \right\} >0.
\end{equation}
The existence of a function $\tilde{\psi}$ satisfying those assumptions is easily obtained for $L$ large enough, e.g. $L =  \hbox{diam}( \Omega) + 5| \overline u| T$: it suffices to choose $\Psi$ taking values in $[0,1]$ and having its critical points in some $\omega_0$ such that $\mbox{dist}(\omega_0, \T_{L} \setminus \mbox{Supp\,} \chi_0) \geq 2 |\overline{u}| T$ and then to propagate $\tilde{\psi}$ with \eqref{TransportPoids2}. This leaves room to define $\omega$. \par
Now once $\tilde{\psi}$ is set, we define
\begin{equation} \label{PsiGamma}
\psi(t,x):=  \tilde \psi(t,x) + 6.
\end{equation}
Next we pick $T_0>0,\ T_1>0$ and $\varepsilon >0$ small enough so that
\begin{equation} \label{Choice-t-0-t-1-eps}
	  T_0+2 T_1 < T - \frac{L_{0} + 12 \varepsilon}{|\overline{u}|}.
\end{equation}
Now for any $\alpha \geq 2$, we introduce the weight function in time $ \theta(t)$ defined by
\begin{equation} \label{ThetaMu}
\theta = \theta(t) \ \mbox{ such that} \
\left\{ \begin{array}{l}
\ds \forall t \in [0,T_0],\, \theta(t) = 1+ \left( 1- \frac{t}{T_0} \right)^\alpha, \smallskip \\
\ds \forall t \in [T_0, T- 2T_1], \, \theta(t) = 1, \smallskip \\
\ds \forall t \in [T-T_1,T), \,  \theta(t) = \frac{1}{T-t}, \smallskip \\
\ds \theta \hbox{ is increasing on } [T-2T_1, T-T_1], \smallskip \\
\ds \theta \in C^2([0,T)).
\end{array} \right.
\end{equation}
Then we consider the following weight function $\varphi = \varphi(t,x)$:
\begin{equation} \label{Phi}
\varphi(t,x)= \theta(t) \left(\lambda_{0} e^{12 \lambda_{0}}- \exp(\lambda_{0} \psi(t,x)) \right),
\end{equation}
where  $s,\, \lambda_{0}$ are positive parameters with $s\geq 1$, $\lambda_{0} \geq 1$ and $\alpha$ is chosen as
\begin{equation} \label{Def-mu}
\alpha = s \lambda_{0}^2 e^{2\lambda_{0}},
\end{equation}
which is always larger than $2$, thus being compatible with the condition $\theta \in C^2([0,T))$. Actually, in the sequel we will use that $s$ can be chosen large enough, but for what concerns $\lambda_{0}$, it can be fixed from the beginning as equal to the constant $\lambda_{0}$ obtained in Theorem~\ref{CarlemanThm-L2L2} below. \par
Also note that, due to the definition of $\psi$ in \eqref{PsiGamma}, to the condition \eqref{Psi} and to $\lambda_{0} \geq 1$, we have for all $(t,x) \in [0,T) \times \T_L$,
\begin{equation} \label{Phi-bounds}
	\frac{14}{15} \Phi(t) \leq \varphi(t,x) \leq \Phi(t),
\end{equation}
where
\begin{equation} \label{DefPhi}
\Phi(t) := \theta(t) \lambda_{0} e^{12 \lambda_{0}}.
\end{equation}
We emphasize that the weight functions $\theta$ and $\varphi$ depend on the parameters $s$ and $\lambda_0$ but we will omit these dependences in the following for simplicity of notations.
\smallskip
\\
\noindent{\bf Notations.} In the following, we will consider functional spaces depending on the time and space variables. This introduces heavy notations, that we will keep in the statements of the theorems, but that we shall simplify in the proof by omitting the time interval $(0,T)$ and the spatial domain $\T_L$ as soon as no confusion can occur. Thus, we will use the notations:
\begin{equation*}
	\norm{\cdot}_{L^2(L^2)} = \norm{\cdot}_{L^2(0,T; L^2(\T_L))}, \quad 	\norm{\cdot}_{L^2(H^2)} = \norm{\cdot}_{L^2(0,T; H^2(\T_L))}, 
\end{equation*}
and so on for the other functional spaces. Similarly, we will often denote by $\norm{\cdot}_{H^2}$ and $\norm{\cdot}_{H^3}$ the norms $\norm{\cdot}_{H^2(\T_L)},\, \norm{\cdot}_{H^3(\T_L)}.$
%
%
%
%%%%%%%%%%%%%%%%%%%%%%%%%%%%%%%%%%%%%%%%%%%%%%%%%%%%%%%%%%%%%%%%%%%%%%%%%%%%%%%%%%%%%%%%%%%%
%
%
%
%
%
\section{The controllability problem (\ref{Control-ReducedEq})}\label{Sec-Control-Sous-System}
The goal of this section is to solve the controllability problem \eqref{Control-ReducedEq}:
\begin{theorem} \label{Thm-Main-Control-Reduced}
	Let $(\overline u,T, \varepsilon)$ be as in \eqref{Choice-t-0-t-1-eps}.
	\\
	Let $(r_0, y_0) \in L^2(\T_L) \times L^2(\T_L)$. There exist $C>0$ and $s_0\geq 1$ large enough such that 
	for all $s \geq s_0$, 
	if $f_r$ and $f_y$ satisfy the integrability conditions
	\begin{equation} \label{Conditions-f-L2}
		\norm{\theta^{-3/2}  f_r e^{s \varphi}}_{L^2(0,T;L^2(\T_L))}+ \norm{ \theta^{-3/2} f_y e^{s \varphi}}_{L^2(0,T; L^2(\T_L))} < \infty,
	\end{equation}
there exists a controlled trajectory $(r,y)$ solving \eqref{Control-ReducedEq} and satisfying the following estimate:
	\begin{multline} \label{Estimate-On-r-y}
	\norm{\theta^{-3/2}  r e^{s \varphi}}_{L^2(0,T;L^2(\T_L))} 
	+ s \norm{ y e^{s \varphi}}_{L^2(0,T;L^2(\T_L))} 
	+   \norm{\theta^{-1} \nabla y e^{s \varphi}}_{L^2(0,T;L^2(\T_L))} 
	\\
	+ \norm{\theta^{-3/2}  \chi_0 v_r e^{s \varphi}}_{L^2(0,T;L^2(\T_L))} 
	+ s^{-1/2} \norm{\theta^{-3/2}  \chi_0 v_y e^{s \varphi}}_{L^2(0,T;L^2(\T_L))} 
	\\
		\leq
	C \left( \norm{\theta^{-3/2}f_r e^{s \varphi}}_{L^2(0,T;L^2(\T_L))} + s^{-1/2} \norm{\theta^{-3/2} f_y e^{s \varphi}}_{L^2(0,T;L^2(\T_L))}	\right.
	\\
	\left.+ \norm{ r_0 e^{ s \varphi(0)}}_{L^2(\T_L)} +  \norm{ y_0 e^{ s \varphi(0)}}_{L^2(\T_L)}\right).
	\end{multline}
Besides, if $(r_0, y_0) \in H^2 (\T_L) \times H^3(\T_L)$, and $f_r$ and $f_y$ satisfy
\begin{equation}
	\label{Cond-Reg-f}
		f_r e^{s \Phi}\in L^2(0,T; H^2(\T_L)), \quad f_y e^{s \Phi} \in L^2(0,T; H^2(\T_L)), 
\end{equation}
we furthermore have the following estimate:
\begin{multline}
	\label{Est-reg-fix-point}
	 	 \norm{r e^{6 s \Phi/7}}_{L^2(0,T;H^2(\T_L))}+ \norm{y e^{6 s \Phi/7 }}_{L^2(0,T; H^4(\T_L))}
\\
	+ \norm{\chi_0 v_r e^{6 s \Phi/7}}_{L^2(0,T;H^2(\T_L))}+ \norm{\chi_0 v_y e^{6 s \Phi/7}}_{L^2(0,T; H^2(\T_L))}
\\	\leq 
	C \left(\norm{f_r e^{s \Phi}}_{L^2(0,T; H^2(\T_L))}+ \norm{f_y e^{s \Phi}}_{L^2(0,T; H^2(\T_L))} + \norm{ r_0 e^{ s \Phi(0)}}_{H^2(\T_L)} +  \norm{ y_0 e^{ s \Phi(0)}}_{H^3(\T_L)}\right), 
\end{multline}
for some constant $C$ independent of $s\geq s_0$.
\end{theorem}
As explained in Section \ref{Sec-Strategy}, Theorem \ref{Thm-Main-Control-Reduced} will be proved by a fixed point theorem based on the understanding of the controllability problem  \eqref{Control-Hyper-ReducedEq}, which amounts to understand two independent controllability problems, one for the heat equation satisfied by $y$, the other one for the transport equation satisfied by~$r$.  \par
The section is then organized as follows. Firstly, we recall the controllability properties of the heat equation. Secondly, we explain how to exhibit a null-controlled trajectory for the transport equation. Thirdly, we explain how these constructions can be combined in order to get Theorem \ref{Thm-Main-Control-Reduced}.
%
%
%
%
%%%%%%%%%%%%%%%%%%%%%%%%%%%%%%%%%%%%%%%%%%%%%%%%%%%%%%
%
%
%
\subsection{Controllability of the heat equation}\label{Subsec-heat}
In this paragraph we deal with the following controllability problem: given $y_0$ and $\tilde f_y$, find a control function $v_y$ such that the solution $y$ of 
\begin{equation} \label{Heat-control}
	\left\{
		\begin{array}{ll}
			\ds \frac{\overline{\rho}}{\nu} \partial_t  y - \Delta y = \tilde{f}_{y}  + v_y {\chi}_{0}, \quad & \hbox{ in }(0,T) \times \T_L,
			\\
			y(0,\cdot) = y_0, \quad & \hbox{ in } \T_L,
		\end{array}
	\right.
\end{equation}
satisfies
\begin{equation}\label{Null-Control-Req}
y(T,\cdot)=0,\quad\hbox{in }\T_L.
\end{equation}
\subsubsection{Results}
As it is done classically, the study of the controllability properties of \eqref{Heat-control} is based on the observability of the adjoint system, which is obtained with the following Carleman estimate:
\begin{theorem}[Theorem 2.5 in \cite{BEG}] \label{CarlemanThm-L2L2}
There exist constants $C_0>0$ and $s_0\geq 1$ and $\lambda_{0}\geq 1$ large enough such that for all smooth functions $w$ on $[0,T] \times \T_L$ and for all $s \geq s_0$,
\begin{multline} \label{CarlemanEst-L2L2}
		s^{3/2} \lambda_0^2 \norm{\xi^{3/2} w e^{- s \varphi}}_{L^2(0,T; L^2(\T_L))}
		+ 
		s^{1/2} \lambda_0 \norm{\xi^{1/2} \nabla w e^{- s \varphi}}_{L^2(0,T; L^2(\T_L))}
		+
		s \lambda_0^{3/2} e^{7 \lambda_0} \norm{ w(0) e^{-s \varphi(0)}}_{L^2(\T_L)}
	\\
		\leq
		C_0 \norm{\left(- \frac{\overline{\rho}}{\nu} \partial_t - \Delta\right) w e^{- s \varphi}}_{L^2(0,T; L^2(\T_L))}
		+
		C_0 s^{3/2} \lambda_0^2 \norm{\xi^{3/2} \chi_0 w e^{- s \varphi}}_{L^2(0,T; L^2(\T_L))}.
\end{multline}
where we have set
\begin{equation} \label{Xi}
	\xi(t,x) = \theta(t) \exp(\lambda_{0} \psi(t,x)) .
\end{equation}
\end{theorem}
Using Theorem \ref{CarlemanThm-L2L2} and the remark that for some constant $C\geq 1$ independent of $s$, 
$$
	\frac{\theta(t)}{C} \leq \xi(t,x) \leq C \theta(t), \quad \text{for all } (t,x) \in [0,T) \times \T_L, 
$$
we obtain the following controllability result:
\begin{theorem}[Inspired by Theorem 2.6 in \cite{BEG}] \label{Thm-Est-Y-Carl-NormsL2}
There exist positive constants $C>0$ and $s_0 \geq 1$ such that for all $s \geq s_0$, for all $\tilde{f}_{y}$ satisfying
\begin{equation}
	\label{Cond-hat-f-y}
	\norm{ \theta^{-3/2} \tilde f_y e^{s \varphi}}_{L^2(0,T; L^2(\T_L))} < \infty
\end{equation}
 and $y_0 \in L^2(\T_L)$, there exists a solution $(y,v_y)$ of the control problem \eqref{Heat-control}--\eqref{Null-Control-Req} which furthermore satisfies the following estimate:
	\begin{multline} \label{Est-Y-L2-H2}
		s^{3/2}  \norm{y e^{ s \varphi}}_{L^2(0,T; L^2(\T_L))}
		+
		\norm{\theta^{-3/2} \chi_0 v_y e^{ s \varphi}}_{L^2(0,T; L^2(\T_L))}
		+
		s^{1/2}  \norm{\theta^{-1} \nabla y e^{ s \varphi}}_{L^2(0,T; L^2(\T_L))}
		\\
		\leq 
		C \norm{\theta^{-3/2} \tilde f_y e^{ s \varphi}}_{L^2(0,T; L^2(\T_L))}
		+
		C s^{1/2} \norm{y_0 e^{s \varphi(0)}}_{L^2(\T_L)}.
	\end{multline}
	Besides, this solution $(y,v_y)$ can be obtained through a linear operator in $(y_0, \tilde f_y)$.
	\\
	If $y_0 \in H^1(\T_L)$, we also have 
	\begin{multline}
		\label{Est-Y-L2H2-Init-H1}
		s^{-1/2} \norm{\theta^{-2} \nabla^2 y e^{ s \varphi}}_{L^2(0,T; L^2(\T_L))}
		\leq
		C \norm{\theta^{-3/2} \tilde f_y e^{ s \varphi}}_{L^2(0,T; L^2(\T_L))}
		\\
		+
		C s^{1/2} \norm{y_0 e^{s \varphi(0)}}_{L^2(\T_L)} + C s^{-1/2} \norm{\nabla y_0 e^{s \varphi(0)}}_{L^2(\T_L)}.
	\end{multline}
\end{theorem}
The proof of Theorem \ref{Thm-Est-Y-Carl-NormsL2} is done in \cite{BEG} for an initial data $y_0 = 0$. We shall therefore not provide extensive details for its proof, but only explain how it should be adapted to the case $y_0 \neq 0$, see the proof in Section \ref{Subsubsec-proof}.
\par
We now explain what can be done when the source term $\tilde f_y$ is more regular and lies in $L^2(0,T; H^1(\T_L))$ or in $L^2(0,T; H^2(\T_L))$.
\begin{proposition}
	\label{Prop-Reg-heat}
	Consider the controlled trajectory $(y,v_y)$ constructed in Theorem \ref{Thm-Est-Y-Carl-NormsL2}. Then for some constant $C >0 $ independent of $s$, we have the following properties:
	\begin{enumerate}
		\item $v_y \in L^2(0,T; H^2(\T_L))$ and 
		$$
			\norm{\chi_0 v_y e^{ 6 s \Phi/7}}_{L^2(0,T; H^2(\T_L))} \leq C \left( \norm{\theta^{-3/2} \tilde f_y e^{s \varphi}}_{L^2(0,T; L^2(\T_L))}+  \norm{ y_0 e^{ s \Phi(0)}}_{L^2(\T_L)}\right).
		$$
		\item If $y_0 \in H^2(\T_L)$ and $\tilde f_y e^{6 s \Phi/7} \in L^2(0,T; H^1(\T_L)),\, \theta^{-3/2} \tilde f_y e^{s \varphi} \in L^2(0,T; L^2(\T_L))$, 
		\begin{multline*}
			\norm{y e^{6 s \Phi/7}}_{L^2(0,T; H^3(\T_L))} \leq  C \left( \norm{\tilde f_y e^{6 s \Phi/7}}_{L^2(0,T; H^1(\T_L))} \right.
			\\
			\left.
			+ \norm{\theta^{-3/2} \tilde f_y e^{s \varphi}}_{L^2(0,T; L^2(\T_L))} +  \norm{ y_0 e^{ s \Phi(0)}}_{H^2(\T_L)}\right). 
		\end{multline*}
		\item If $y_0 \in H^3(\T_L)$ and $\tilde f_y e^{6 s \Phi/7} \in L^2(0,T;H^2(\T_L)),\,  \theta^{-3/2} \tilde f_y e^{s \varphi} \in L^2(0,T; L^2(\T_L))$, 
		\begin{multline*}
			\norm{y e^{6 s \Phi/7}}_{L^2(0,T; H^4(\T_L))} \leq  C \left( \norm{\tilde f_y e^{6 s \Phi/7}}_{L^2(0,T; H^2(\T_L))}\right.
			\\
			\left.+\norm{\theta^{-3/2} \tilde f_y e^{s \varphi}}_{L^2(0,T; L^2(\T_L))}+ \norm{ y_0 e^{ s \Phi(0)}}_{H^3(\T_L)}\right). 
		\end{multline*}
	\end{enumerate}
\end{proposition}
The proof is done below in Section \ref{Subsubsec-proof} and is mainly based on regularity results.
\subsubsection{Proofs}\label{Subsubsec-proof}
\begin{proof}[Sketch of the proof of Theorem \ref{Thm-Est-Y-Carl-NormsL2}]
	For later purpose, let us briefly present how the proof of Theorem~\ref{Thm-Est-Y-Carl-NormsL2} works. It mainly consists in introducing the functional 
	\begin{multline}
		\label{Func-J}
		J(w) =  \frac{1}{2} \int_0^T \int_{\T_L} |(-  \frac{\overline{\rho}}{\nu} \partial_t - \Delta) w|^2 e^{-2 s \varphi} 
			+  \frac{s^3}{2}\int_0^T \int_{\T_L}  \chi_{0}^2 \theta^3 |w|^2 e^{-2 s \varphi}
			\\
			- \int_0^T \int_{\T_L} \tilde f_y w + \int_{\T_L} w(0,\cdot) y_0(\cdot), 
	\end{multline}
	considered on the set
	\begin{equation}
		\label{CarlemanSpace-Xobs}
		H_{obs} = \overline{\{ w \in C^\infty([0,T] \times \T_L) \}}^{\norm{\cdot }_{obs}}.
	\end{equation}
Here the overline refers to the completion with respect to the Hilbert norm $\norm{\cdot}_{obs}$ defined by 
	\begin{equation}
		\label{CarlemanNorm-Xobs}
		\norm{w}_{obs}^2 =  
			\int_0^T \int_{\T_L} |(- \frac{\overline{\rho}}{\nu} \partial_t - \Delta) w|^2 e^{-2 s \varphi} 
			+  s^3 \int_0^T \int_{\T_L} \chi_{0}^2 \theta^3 |w|^2 e^{-2 s \varphi}.
	\end{equation}

	Thanks to the Carleman estimate \eqref{CarlemanEst-L2L2}, $\norm{\cdot}_{obs}$ is a norm. The assumptions $y_0 \in L^2$ and \eqref{Cond-hat-f-y}
 imply that $J$ is well-defined, convex and coercive on $H_{obs}$. Therefore it has a unique minimizer $W$ in $H_{obs}$ and the couple $(y,v_y)$ given by 
	\begin{equation}
		\label{Link-W-Y}
		y = e^{-2s \varphi} (- \frac{\overline{\rho}}{\nu} \partial_t - \Delta) W, \qquad v_y = - s^3 \theta^3 \chi_{0} W e^{-2s \varphi}  
	\end{equation}
	solves the controllability problem \eqref{Heat-control}--\eqref{Null-Control-Req}. Using the coercivity of $J$ immediately yields $L^2(L^2)$ estimates on $y$ and $v_y$ and on $\norm{W}_{obs}$ by using $J(W) \leq J(0) = 0$:
	\begin{multline}
		s^3 \int_0^T \int_{\T_L}  |y|^2 e^{2 s \varphi}+ \int_0^T \int_{\T_L} \theta^{-3} |v_y|^2 e^{2 s \varphi} 
		 =  
		s^3 \norm{W}_{obs}^2 
		\\
		 \leq  C s^3 
			\left( 
				\frac{1}{s^3 } \int_0^T \int_{\T_L} \theta^{-3} |\tilde f_y|^2 e^{2 s \varphi}  
				+
				 \frac{1}{s^2 } \int_{\T_L}  |y_0|^2 e^{2 s \varphi(0)}
			 \right). 
			 \label{Estimate-W-obs}
	\end{multline}
	 The other estimates on $y$ are derived by  weighted energy estimates similar to the ones developed in \cite[Theorem 2.6]{BEG}, the only difference coming from the integrations by parts in time leading to new terms involving $y_0$. Details of the proof are left to the reader.
\end{proof}
\begin{proof}[Proof of Proposition \ref{Prop-Reg-heat}]
	\emph{Item 1.} The control $v_y$ is given by \eqref{Link-W-Y} with $W \in X_{obs}$ with an estimate on $\norm{W}_{obs}$ given by \eqref{Estimate-W-obs}. Therefore, $v_y e^{6 s \Phi/7} = s^3 \theta_3 \chi_0 W e^{6 s \Phi/7 - 2 s \varphi}$. We look at the equation satisfied by $W_* = e^{-106 s \Phi /105 } W$:
	$$
		(- \frac{\overline{\rho}}{\nu} \partial_t - \Delta) W_* = e^{-106 s \Phi/105 } (- \frac{\overline{\rho}}{\nu} \partial_t - \Delta) W + \frac{106}{105} s\frac{\overline{\rho}}{\nu} \partial_t \Phi e^{-106 s \Phi/105 } W, 
	$$
	and $W_*(T) = 0$ (in $\mathscr{D}'(\T_L)$). Using \eqref{Estimate-W-obs} and \eqref{CarlemanEst-L2L2}, we get an $L^2(L^2)$ bound on the right hand-side since $\varphi \leq \Phi$. Maximal regularity estimates then yield $W_* = W e^{-106 s \Phi/105} \in L^2(H^2)$. From $6 \Phi/7- 2\varphi \leq - 106\Phi/105$, see \eqref{Phi-bounds}, we thus get the claimed estimates.
	\\
	\emph{Items 2 and 3.} Let us give some partial details on the proof of item 2. We set $y_* = y e^{6 s \Phi/7}$. It solves:
	$$
		\frac{\overline{\rho}}{\nu} \partial_t y_*  - \Delta y_* = e^{6 s \Phi/7} (\tilde f_y + v_y) + \frac{6}{7} s\frac{\overline{\rho}}{\nu} \partial_t \Phi e^{6 s \Phi/7 } y. 
	$$
	We then simply use classical parabolic regularity estimates for $y_*$, item 1 and \eqref{Est-Y-L2-H2}--\eqref{Est-Y-L2H2-Init-H1}. 
\end{proof}
%
%%%%%%%%%%%%%%%%%%%%%%%%%%%%%%%%%%%%%%%%%%%%%%%%%%
%
%
%
\subsection{Controllability of the transport equation}
We study the following control problem: Given $\tilde f_r$ and $r_0$, find a control function $v_r$ such that the solution $r$ of
\begin{equation} \label{Eq-Trans-For-r}
\left\{
	\begin{array}{ll}
		\ds \partial_t r + \overline{u} \cdot \nabla r + \frac{p'(\overline{\rho}) \overline{\rho}}{\nu} r = \tilde f_r+ v_r \chi_0, & \text{ in } (0,T) \times \T_L, 
		\\
		r(0,\cdot) = r_0, &  \text{ in } \T_L, 
	\end{array}
\right. 
\end{equation}
satisfies the controllability requirement
\begin{equation} \label{r-T}
r(T, \cdot ) = 0 \text{ in } \T_L.
\end{equation}
We show the following existence result:
\begin{theorem} \label{Thm-Trans-Control}
	Let $(\overline u,T, \varepsilon)$ be as in \eqref{Choice-t-0-t-1-eps}.
	\\
	For all $\tilde f_r$ with 
		\begin{equation} \label{Cond-hat-f-r}
			\norm{\theta^{-3/2} \tilde f_r e^{s \varphi}}_{L^2(0,T; L^2(\T_L))} <\infty
		\end{equation}
	 and $r_0 \in L^2(\T_L)$, there exists a function $v_r \in L^2(0,T;L^2(\T_L))$ such that the solution $r$ of \eqref{Eq-Trans-For-r} satisfies the control requirement \eqref{r-T}. 
	Besides, the controlled trajectory $r$ and the control function $v_r$ satisfy
	\begin{multline} \label{Est-r-L2}
		\norm{\theta^{-3/2} r e^{s \varphi}}_{L^2(0,T; L^2(\T_L))} +\norm{\theta^{-3/2} v_r e^{s \varphi}}_{L^2(0,T; L^2(\T_L))} 
		\\
		\leq 
		C 
		\left(\norm{\theta^{-3/2} \tilde f_r e^{s \varphi}}_{L^2(0,T; L^2(\T_L))} 
		+
		\norm{r_0 e^{s \varphi(0)}}_{L^2(\T_L)}\right).
	\end{multline}
	If $r_0 \in H^1(\T_L)$ and $\tilde f_r$ satisfies \eqref{Eq-Trans-For-r} and $ \tilde f_r e^{6s \Phi/7} \in L^2(0,T;H^1(\T_L))$, then $r$ furthermore belongs to $L^2(0,T;H^1(\T_L))$ and satisfies
	\begin{multline}
		\label{Est-r-H1}
		\norm{ r e^{6 s \Phi/7} }_{L^2(0,T; H^1(\T_L))}+ \norm{  v_r e^{6 s \Phi/7} }_{L^2(0,T; H^1(\T_L))}
		\\
		\leq
		C 
		\left(
		\norm{\tilde f_r e^{6 s \Phi/7} }_{L^2(0,T;H^1(\T_L))}
		+
		\norm{\theta^{-3/2} \tilde f_r e^{s \varphi} }_{L^2(0,T;L^2(\T_L))}
		+
		\norm{r_0 e^{s \Phi(0)} }_{H^1(\T_L)}
		\right).
	\end{multline}
	If $r_0 \in H^2(\T_L)$ and $\tilde f_r$ satisfies \eqref{Eq-Trans-For-r} and $ \tilde f_r e^{6 s \Phi/7} \in L^2(0,T;H^2(\T_L))$, then $r$ belongs to $L^2(0,T;H^2(\T_L))$ and satisfies
	\begin{multline}
		\label{Est-r-H2}
		\norm{r e^{6 s \Phi/7} }_{L^2(0,T; H^2(\T_L))}+ \norm{ v_r e^{6 s \Phi/7}}_{L^2(0,T; H^2(\T_L))}
		\\
		\leq
		C \left(
		\norm{ \tilde f_r e^{6 s \Phi/7} }_{L^2(0,T; H^2(\T_L))}
		\right.
		+
		\left.
		\norm{\theta^{-3/2} \tilde f_r e^{s \varphi} }_{L^2(0,T;L^2(\T_L))}
		+
		\norm{r_0 e^{s \Phi(0)} }_{H^2(\T_L)}\right).
	\end{multline}
	Besides, this solution $(r, v_r)$ can be obtained through a linear operator in $(r_0, \tilde f_r)$.
\end{theorem}
\begin{proof}
The proof of Theorem \ref{Thm-Trans-Control} consists in an explicit construction solving the control problem \eqref{Eq-Trans-For-r}--\eqref{r-T} and then on suitable estimates on it. \par
\ \\
\noindent
{\it An explicit construction.} Let $\eta_0$ be a smooth cut-off function taking value $1$ on $\{x \in \T_L, \hbox{ with } d(x, \overline{\Omega})\leq 5 \varepsilon + |\overline u| T_0 \}$ and vanishing on $\{ x \in \T_L, \, d(x, \Omega) \geq 6 \varepsilon+ |\overline u| T_0\}$. We then introduce $\eta$ the solution of 
	\begin{equation} \label{Def-eta}
		\left\{
			\begin{array}{l}
				\partial_t \eta + \overline{u} \cdot \nabla \eta  = 0, \text{ in } (0,T) \times \T_L,	
				\\
				\eta(0, \cdot) = \eta_0 \text{ in } \T_L,
			\end{array}
		\right.
	\end{equation}
	and the solutions $r_f$ and $r_b$ (here `$f$' stands for forward, `$b$' for backward) of
	\begin{equation} \label{Def-r-f}
		\left\{
			\begin{array}{l}
				\partial_t r_f + \overline{u} \cdot \nabla r_f + a r_f =  \tilde f_r, \text{ in } (0,T) \times \T_L,	
				\\
				r_f(0, \cdot) = r_0 \text{ in } \T_L,
			\end{array}
		\right.
	\end{equation}
	and
	\begin{equation} \label{Def-r-b}
		\left\{
			\begin{array}{l}
				\partial_t r_b + \overline{u} \cdot \nabla r_b + a r_b = \tilde f_r, \text{ in } (0,T) \times \T_L,	
				\\
				r_b(T, \cdot) = 0 \text{ in } \T_L,
			\end{array}
		\right.
	\end{equation}
	where $a$ denotes the constant
	$$
		a= \frac{p'(\overline{\rho}) \overline{\rho}}{\nu}.
	$$
	We then set
	\begin{equation} \label{r-Explicit-Construction}
		r = \eta_2(x) \left( \eta r_f + (1- \eta) r_b \right) + ( 1- \eta_2(x)) \eta_1(t) r_f, \text{ in } (0,T) \times \T_L, 
	\end{equation}
	where $\eta_1(t)$ is a smooth cut-off function taking value $1$ on $[0,T_0/2]$ and vanishing for $t \geq T_0$ and $\eta_2 = \eta_2(x)$ is a smooth cut-off function taking value $1$ for $x$ with $d(x, \Omega) \leq 3 \varepsilon$ and vanishing for $x$ with $d(x, \Omega)\geq 4 \varepsilon$.
	One easily checks that $r$ solves
	\begin{multline}
		\partial_t r +\overline{u} \cdot \nabla r + a r = \eta_2 \tilde f_r + (1- \eta_2) \eta_1 \tilde f_r + \overline{u} \cdot \nabla \eta_2  (\eta r_f + (1- \eta) r_b)
		\\
		-   \eta_1 \overline{u} \cdot \nabla \eta_2  r_f + (1- \eta_2) \partial_t \eta_1 r_f
		\text{ in } (0,T) \times \T_L,
	\end{multline}
	thus corresponding to a control function
\begin{equation} \label{v-explicit}
	v_r = 
	(\eta_2-1) \tilde f_r + (1- \eta_2) \eta_1 \tilde f_r + \overline{u}  \cdot \nabla \eta_2  (\eta r_f + (1- \eta) r_b)
		\\
		-   \eta_1 \overline{u} \cdot \nabla \eta_2  r_f + (1- \eta_2) \partial_t \eta_1 r_f, 
\end{equation}
	localized in the support of $\chi_0$ due to the condition on the support of $\eta_2$.
	Besides, $r$ given by \eqref{r-Explicit-Construction} satisfies
	$$
		r(0, \cdot) =  r_0 \text{ in } \T_L, \qquad r(T, \cdot) = 0 \text{ in } \T_L
	$$
	due to the conditions on the support of $\eta_0$, $\eta_1$, $\eta_2$ and the condition \eqref{Choice-t-0-t-1-eps} on the flow corresponding to $\overline{u}$.
	\\
	Actually, thanks to the choice of $\varepsilon>0, T_0>0$ and $T_1>0$ in \eqref{Choice-t-0-t-1-eps} we have
	\begin{equation}
		\label{Supports-Eta}
		\eta_2 (1 - \eta) = 0 \quad \hbox{ for all } (t,x) \in [0,T_0] \times \T_L,\quad \hbox{ and } \quad \eta_2 \eta = 0 \quad \hbox{ for all } (t,x) \in [T-2T_1,T] \times \T_L.
	\end{equation}
\par
\ \\
\noindent
{\it Estimates on $r$.} 
Let us start with estimates on $r_f$. 
To get estimates on $r_f$, we perform weighted energy estimates on \eqref{Def-r-f} on the time interval $(0,T - 2T_1)$. Multiplying \eqref{Def-r-f} by $\theta^{-3} r_f e^{2 s \varphi}$, we obtain
\begin{multline}
	\frac{d}{d} \left( \frac{1}{2} \int_{\T_L} \theta^{-3} |r_f|^2 e^{2s \varphi} \right) \leq \frac{1}{2} \int_{\T_L} |r_f|^2 \left(  - 2 a \theta^{-3} e^{2 s \varphi} + (\partial_t + \overline{u}\cdot \nabla) (\theta^{-3} e^{2s \varphi})\right)
	\\
	+ \left(\int_{\T_L} \theta^{-3} |r_f|^2 e^{2s \varphi} \right)^{1/2} \left(\int_{\T_L} \theta^{-3} |\tilde f_r|^2 e^{2s \varphi} \right)^{1/2}. 
\end{multline}
But, for all $t \in (0,T - 2 T_1)$ and $x \in \T_L$,
$$
	(\partial_t + \overline{u} \cdot \nabla) (\theta^{-3} e^{2s \varphi}) \leq 0. 
$$
We thus conclude
$$
	\norm{\theta^{-3/2} r_f  e^{s \varphi}}_{L^\infty(0,T-2 T_1; L^2)}
	\leq 
	C 
	\norm{ \theta^{-3/2} \tilde f_r e^{s \varphi}}_{L^2(0,T-2T_1; L^2)} 
	+
	C \norm{r_0 e^{s \varphi(0)}}_{L^2}.
$$
Similarly, one can show that $r_b$ satisfies
 $$
	\norm{\theta^{-3/2}r_b  e^{s \varphi}}_{L^\infty(T_0,T; L^2)}
	\leq 
	C 
	\norm{ \theta^{-3/2}\tilde f_r e^{s \varphi}}_{L^2(T_0,T; L^2)}. 
$$	
To conclude that 
$$
	\norm{\theta^{-3/2} re^{s \varphi}}_{L^\infty(0,T; L^2)}
	\leq 
	C 
	\norm{ \theta^{-3/2}\tilde f_r e^{s \varphi}}_{L^2(0,T; L^2)} 
	+
	C \norm{r_0 e^{s \varphi(0)}}_{L^2}, 
$$
we use the explicit definition of $r$ in \eqref{r-Explicit-Construction} and identity  \eqref{Supports-Eta}, and notice that $\eta_0$, $\eta_1$ and $\eta_{2}$ belong to $L^\infty,$ and $(\eta, \widehat  u)$ to $L^\infty(L^{\infty})$.
\par
The estimate on $v_r$ in \eqref{Est-r-L2} is also a simple consequence of its explicit value in \eqref{v-explicit} and the fact that $\eta_0 \in W^{1,\infty},\, \eta_1 \in W^{1,\infty},\,  \eta \in L^\infty(L^{\infty}),\,  \eta_2 \in L^\infty$.\par
\ \\
\noindent {\it Regularity results.}
	To obtain regularity results on $r$ and $v_r$, it is then sufficient to get regularity estimates on $r_f$  solution of \eqref{Def-r-f} on the time interval $(0,T-2 T_1)$ and on $r_b$ solution of \eqref{Def-r-b} on the time interval $(T_0, T)$. Of course, these estimates will be of the same nature, so we only focus on $r_f$, the other case being completely similar.  
	\par
	To get weighted estimates in higher norms, we do higher order energy estimates on \eqref{Def-r-f}. For instance, $\nabla r_f$ satisfies the equation
	\begin{equation} \label{Def-Nabla-r-f}
		\left\{
			\begin{array}{l}
				\partial_t \nabla r_f + (\overline{u} \cdot \nabla) \nabla r_f + a \nabla r_f  =  \nabla \tilde f_r, \text{ in } (0,T) \times \T_L,	
				\\
				\nabla r_f(0, \cdot) = \nabla r_0 \text{ in } \T_L,
			\end{array}
		\right.
	\end{equation}
	Hence, using that $\partial_t \Phi \leq 0$ on $(0,T- 2T_1)$, energy estimates directly provide
	\begin{equation}
		\label{Est-r-H1-r-f}
		\norm{\nabla r_f e^{6 s \Phi/7}}_{L^\infty(0,T-2 T_1; L^2)} \leq C
		\left( \norm{\nabla \tilde f_r e^{6 s \Phi/7}}_{L^1(L^2)} + \norm{\nabla r_0 e^{6 s \Phi/7}}_{L^2} \right).
	\end{equation}
	This implies \eqref{Est-r-H1}.
	\par
	The equation of $\nabla^2 r_f$ has the same form. For all $(i,j) \in \{1, \cdots, d\}^2$,
	\begin{equation} \label{Def-Nabla-2-r-f}
		\left\{
			\begin{array}{l}
				\partial_t \partial_{i,j} r_f + (\overline{u} \cdot \nabla) \partial_{i,j} r_f + a \partial_{i,j} r_f   =  \partial_{i,j} \tilde f_r, \text{ in } (0,T) \times \T_L,	
				\\
				\partial_{i,j} r_f(0, \cdot) = \partial_{i,j} r_0 \text{ in } \T_L. 
			\end{array}
		\right.
	\end{equation}
	An energy estimate for $\nabla^2 r_f$ on $(0,T-2 T_1)$ directly yields
	$$
		\norm{\nabla^2 r_f e^{6 s \Phi/7}}_{L^\infty(0,T-2 T_1; L^2)} \leq 
		C \left(
		\norm{\nabla^2 \tilde f_r e^{6 s \Phi/7}}_{L^1(L^2)} + \norm{\nabla^2 r_0 e^{6 s \Phi/7}}_{L^2}
		\right),
	$$
	thus concluding the proof of Theorem \ref{Thm-Trans-Control}.
\end{proof}
%
%
%
%%%%%%%%%%%%%%%%%%%%%%%%%%%%%%%%%%%%%%%%%%%%%%%%
%
%
%
\subsection{Proof of Theorem \ref{Thm-Main-Control-Reduced}}
{\it Existence of a solution to the control problem.} We construct the controlled trajectory using a fixed point argument. 

We introduce the sets
\begin{equation}
	\begin{array}{ll}
	&
	\mathscr{C}_{s}^r = \{r \in L^2(0,T; L^2(\T_L))\,  \hbox{ such that }  \theta^{-3/2} r e^{s \varphi} \in L^2(0,T; L^2(\T_L))
	\},
	\smallskip
	\\
	& \mathscr{C}_{s}^y = \{y \in L^2(0,T; H^1(\T_L))\,  \hbox{ such that }  y e^{s \varphi}, \theta^{-1} \nabla y e^{s \varphi} \in L^2(0,T; L^2(\T_L)) \}.
	\end{array}
\end{equation}

For $\tilde r \in \mathscr{C}_{s}^r$ and $\tilde y \in \mathscr{C}_{s}^y$, we introduce
\begin{eqnarray*}
	\tilde f_r : =& \tilde f_r(\tilde y) & =  f_r 
	  - \frac{p'(\overline{\rho})\overline{\rho}^3}{\nu^2} \tilde y, 
	\\
	\tilde f_y :=&\tilde f_y(\tilde r, \tilde y) &= f_y + \frac{p'(\overline{\rho})}{\nu} \tilde r - \frac{\overline{\rho}\overline{u}}{\nu} \cdot \nabla \tilde y + \frac{p'(\overline \rho) \overline{\rho}^2}{\nu^2} \tilde y. 
\end{eqnarray*}
	As $f_r$ and $f_y$ satisfy \eqref{Conditions-f-L2}, for $(\tilde r, \tilde y) \in \mathscr{C}_{s}^r \times \mathscr{C}_{s}^y$, $\tilde f_r$ satisfies \eqref{Cond-hat-f-r} and $\tilde f_y$ satisfies \eqref{Cond-hat-f-y}.\par
Therefore, one can define a map $\Lambda_{s}$ on $\mathscr{C}_{s}^r\times \mathscr{C}_{s}^y$ which to a data $(\tilde r, \tilde y) \in \mathscr{C}_{s}^r\times \mathscr{C}_{s}^y$ associates $(r,y)$, where $r$ is the solution of the controlled problem 
\begin{equation}
\label{Control-Hyper-ReducedEq-FixPoint-r}
\left\{ \begin{array}{ll}
	\ds  \partial_{t} r + \overline{u}  \cdot \nabla r + \frac{p'(\overline{\rho})\overline{\rho} }{\nu} r =   \tilde f_r + v_r \chi_0,
	& \text{ in } (0,T) \times \T_L,  
	\\
	r(0, \cdot) = r_0(\cdot), \qquad r(T, \cdot) = 0, 
	& \text{ in } \T_L,  
\end{array}\right.
\end{equation}
given by Theorem \ref{Thm-Trans-Control}, and $y$ is the solution of the controlled problem
\begin{equation} \label{Control-Hyper-ReducedEq-FixPoint-y}
\left\{ \begin{array}{ll}	
	\ds  \frac{\overline{\rho}}{\nu} \partial_{t} y -  \Delta y  =
	\tilde f_y + v_y \chi_0, 
	& \text{ in } (0,T) \times \T_L, 
	\\
	y(0, \cdot) = y_0(\cdot), \qquad y(T, \cdot) = 0, 
	& \text{ in } \T_L,  
\end{array} \right.
\end{equation}
given by Theorem \ref{Thm-Est-Y-Carl-NormsL2}. \par
Then we remark that Theorems~\ref{Thm-Est-Y-Carl-NormsL2} and \ref{Thm-Trans-Control} both yield a linear construction, respectively for $(y_0, \tilde f_y) \mapsto (y, v_y)$ and for $(r_0, \tilde f_r) \mapsto (r, v_r)$. In order to apply Banach's fixed point theorem, let us show that the map $\Lambda_s$ is a contractive mapping for $s$ large enough. \par
Let $(\tilde r_a, \tilde y_a)$ and $(\tilde r_b, \tilde y_b)$ be elements of $\mathscr{C}_{s}^r\times \mathscr{C}_{s}^y$, and call their respective images  
$(r_a, y_a) = \Lambda_{s}(\tilde r_a, \tilde y_a)$, and $(r_b, y_b) = \Lambda_{s}(\tilde r_b, \tilde y_b)$.
Setting ${\mathcal R} = r_a - r_b$, ${\mathcal Y} = y_a - y_b$, $\tilde{\mathcal R} = \tilde r_a - \tilde r_b$, $\tilde{\mathcal Y} = \tilde y_a - \tilde y_b$,
$ \tilde{\mathcal F}_r = \tilde f_r (\tilde y_a) - \tilde f_r (\tilde y_b)$ and  $\tilde{\mathcal F}_y = \tilde f_y (\tilde r_a, \tilde y_a) - \tilde f_y (\tilde r_b, \tilde y_b)$, by Theorem \ref{Thm-Trans-Control} we have
\begin{equation*}
	\norm{\theta^{-3/2} {\mathcal R} e^{s \varphi}}_{L^2(L^2)} 
	\leq 
	C \norm{\theta^{-3/2} \tilde{\mathcal F}_r e^{s \varphi}}_{L^2(L^2)} 
	\leq
	C \norm{\theta^{-3/2} \tilde{\mathcal Y} e^{s \varphi}}_{L^2(L^2)}, 
\end{equation*}
while Theorem \ref{Thm-Est-Y-Carl-NormsL2} implies
\begin{multline*}
		s^{3/2}  \norm{ {\mathcal Y} e^{s \varphi}}_{L^2(L^2)} 
		+
		s^{1/2} \norm{\theta^{-1} \nabla {\mathcal Y} e^{s \varphi}}_{L^2(L^2)} 
		\leq
		 C \norm{\theta^{-3/2} \tilde{\mathcal F}_{y} e^{s \varphi}} 
		\\
		\leq C \left( 
		\norm{\theta^{-3/2} \tilde{\mathcal R} e^{s \varphi}}_{L^2(L^2)} 
		+ 
		\norm{\theta^{-3/2} \tilde{\mathcal Y} e^{s \varphi}}_{L^2(L^2)}
		+
		 \norm{\theta^{-3/2} \nabla \tilde{\mathcal Y}  e^{s \varphi}}_{L^2(L^2)}
		\right).
\end{multline*}
In particular, we have
\begin{multline*}
	\norm{\theta^{-3/2} {\mathcal R} e^{s \varphi}}_{L^2(L^2)} 
	+
	s  \norm{ {\mathcal Y} e^{s \varphi}}_{L^2(L^2)} 
	+ 
	 \norm{\theta^{-1} \nabla {\mathcal Y} e^{s \varphi}}_{L^2(L^2)}
	\\
	\leq
	C s^{-1/2} \left(  \norm{\theta^{-3/2} \tilde{\mathcal R} e^{s \varphi}}_{L^2(L^2)} + s  \norm{ \tilde{\mathcal Y} e^{s \varphi}}_{L^2(L^2)} 
	+ 
	 \norm{\theta^{-1} \nabla \tilde{\mathcal Y} e^{s \varphi}}_{L^2(L^2)}\right).
\end{multline*}
Thus the quantity
$$
	\norm{(r, y)}_{\mathscr{C}_{s}^r\times \mathscr{C}_{s}^y} = 
	\norm{\theta^{-3/2} r e^{s \varphi}}_{L^2(L^2)} 
	+
	s  \norm{ y e^{s \varphi}}_{L^2(L^2)} 
	+ 
	 \norm{\theta^{-1} \nabla y e^{s \varphi}}_{L^2(L^2)}
$$
defines a norm on $\mathscr{C}_{s}^r\times \mathscr{C}_{s}^y$ for which the map $\Lambda_s$ satisfies
\begin{equation}
	\norm{\Lambda_{s}(\tilde r_a, \tilde y_a)- \Lambda_{s}(\tilde r_b, \tilde y_b)}_{\mathscr{C}_{s}^r\times \mathscr{C}_{s}^y} 
	\leq
	C s^{-1/2} \norm{(r_a ,y_a)- (r_b, y_b)}_{\mathscr{C}_{s}^r\times \mathscr{C}_{s}^y}.
\end{equation}
Consequently, if $s$ is chosen large enough, the map $\Lambda_s$ is a contractive mapping and by Banach's fixed point theorem, $\Lambda_s$ has a unique fixed point $(r, y)$ in $\mathscr{C}_{s}^r\times \mathscr{C}_{s}^y$. 
By construction, this fixed point $(r,y)$ solves the controllability problem \eqref{Control-ReducedEq}. Besides, estimating $\tilde f_r(y)$ and $\tilde f_y(r)$ by 
\begin{equation}
\norm{\theta^{-3/2} \tilde f_r(y) e^{s \varphi}}_{L^2(L^2)} 
\leq 
C 
 \norm{ y e^{s \varphi}}_{L^2(L^2)} 
+
C \norm{\theta^{-3/2}f_r e^{s \varphi}}_{L^2(L^2)} 
\end{equation}
and
\begin{multline}
\norm{\theta^{-3/2} \tilde f_y(r,y) e^{s \varphi}}_{L^2(L^2)} 
\leq 
C \norm{\theta^{-3/2} f_y e^{s \varphi}}_{L^2(L^2)}
\\
+
C \left( \norm{\theta^{-3/2} r e^{s \varphi}}_{L^2(L^2)} 
+
\norm{\theta^{-3/2} y e^{s \varphi}}_{L^2(L^2)}
+ 
\norm{\theta^{-3/2} \nabla y e^{s \varphi}}_{L^2(L^2)}\right)
,  
\end{multline}
one gets with Theorems~\ref{Thm-Est-Y-Carl-NormsL2} and \ref{Thm-Trans-Control} that $(r,y)$ solution of \eqref{Navier-Stokes-Adjoint-Mu-Sigma} satisfies
\begin{multline*}
	\norm{\theta^{-3/2}  r e^{s \varphi}}_{L^2(L^2)} 
	+ s \norm{ y e^{s \varphi}}_{L^2(L^2)} 
	+   \norm{\theta^{-1} \nabla y e^{s \varphi}}_{L^2(L^2)} 
	\\
\leq
C \left( \norm{\theta^{-3/2} f_r e^{s \varphi}}_{L^2(L^2)} 
+
 s^{-1/2} \norm{\theta^{-3/2} f_y e^{s \varphi}}_{L^2(L^2)}\right) 
+
C \left( \norm{ r_0 e^{ s \varphi(0)}}_{L^2} +  \norm{ y_0 e^{ s \varphi(0)}}_{L^2}\right), 
\end{multline*}
that is, the estimate \eqref{Estimate-On-r-y}.\par
\medskip
\noindent{\it Regularity estimates on the solution of the control problem.} We now further assume that $(r_0, y_0) \in H^2(\T_L) \times H^3(\T_L)$ and $f_r, \, f_y$ satisfy \eqref{Cond-Reg-f}. We use a bootstrap argument. For preciseness, we set 
\begin{eqnarray*}
	f_r(y)  & = &  f_r  - \frac{p'(\overline{\rho})\overline{\rho}^3}{\nu^2} y, 
	\\
	f_y(r,y) & = &  f_y + \frac{p'(\overline{\rho})}{\nu} r - \frac{\overline{\rho}\overline{u}}{\nu} \cdot \nabla y + \frac{p'(\overline \rho) \overline{\rho}^2}{\nu^2} y. 
\end{eqnarray*}
First, thanks to \eqref{Est-Y-L2H2-Init-H1}, we have $\theta^{-2} y e^{s\varphi} \in L^2(0,T;H^2(\T_L))$, so that $ f_r(y) e^{6s \Phi/7}  \in L^2(0,T; H^2(\T_L))$  (recall condition \eqref{Cond-Reg-f}). Using \eqref{Est-r-H2}, we deduce $ r e^{6 s \Phi/7} \in L^2(0,T;H^2(\T_L))$. Hence we obtain that $ f_y(r, y) e^{6 s\Phi/7} \in L^2(0,T;H^1(\T_L))$. By Proposition~\ref{Prop-Reg-heat}, $y e^{6 s \Phi/7} \in L^2(0,T; H^3(\T_L))$. Therefore, we have $ f_y(r, y) e^{6 s\Phi/7} \in L^2(0,T;H^2(\T_L))$, and using again Proposition \ref{Prop-Reg-heat}, we get $y e^{6 s \Phi/7} \in L^2(0,T; H^4(\T_L))$.
\\
The above regularity results come with estimates. Tracking them yields estimate \eqref{Est-reg-fix-point}.
\section{Controllability estimates for (\ref{Navier-Stokes-Linear})}\label{Sec-Control-NS-Linear}
In this section, we study the controllability of \eqref{Navier-Stokes-Linear}. This is given in the following statement.
\begin{theorem}
\label{Thm-Control-NS-linear}
	There exists $s_0 \geq 1$, such that  for all $s \geq s_0$, for all $(\widehat  \rho_0, \widehat  u_0) \in H^2(\T_L) \times H^2(\T_L)$, $\widehat  f_\rho, \widehat  f_u$ such that $\widehat  f_\rho e^{s \Phi} \in L^2(0,T; H^2(\T_L)), \, \widehat  f_u e^{7 s \Phi/6 } \in L^2(0,T; H^1(\T_L))$, there exist control functions $v_\rho, \, v_u$ and a corresponding controlled trajectory $(\rho, u)$ solving \eqref{Navier-Stokes-Linear} with initial data $(\widehat  \rho_0, \widehat  u_0)$, satisfying the controllability requirement \eqref{ControlReq}, and depending linearly on the data $(\widehat  \rho_0, \widehat  u_0, \widehat  f_\rho, \widehat  f_u)$. Besides, we have the estimate:
	\begin{multline}
		\label{Estimate-Control-Linear-1}
		\! \! \! \!
		\norm{(\rho e^{6 s  \Phi/7}, u e^{6 s \Phi/7} )}_{L^2(0,T; H^2(\T_L))\times L^2(0,T; H^3(\T_L))} 
		+ 
		\norm{(\chi v_\rho e^{6s \Phi/7}, \chi v_u e^{6 s \Phi/7})}_{L^2(0,T; H^2(\T_L)) \times L^2(0,T; H^1(\T_L))}
	\\
		\leq
		C
		\norm{(\widehat  f_\rho e^{s \Phi}, \widehat  f_u e^{7 s \Phi/6})}_{L^2(0,T; H^2(\T_L)) \times L^2(0,T; H^1(\T_L))} 
		+ 
		C 
		\norm{(\widehat  \rho_0 e^{s \Phi(0)}, \widehat  u_0 e^{7 s \Phi(0)/6} )}_{H^2(\T_L)\times H^2(\T_L)}.
	\end{multline}
	In particular, this implies
	\begin{multline}
		\label{Estimate-Control-Linear-2}
		\norm{(\rho e^{5 s  \Phi/6}, u e^{5 s \Phi/6} )}_{(C^0([0,T]; H^2(\T_L)) \cap H^1 (0,T; L^2(\T_L)))\times (L^2(0,T; H^3(\T_L))\cap C^0([0,T]; H^2(\T_L)) \cap H^1(0,T; H^1(\T_L)))} 
	\\
		\leq
		C
		\norm{(\widehat  f_\rho e^{s \Phi}, \widehat  f_u e^{7 s \Phi/6})}_{L^2(0,T; H^2(\T_L)) \times L^2(0,T; H^1(\T_L))} 
		+ 
		C 
		\norm{(\widehat  \rho_0 e^{s \Phi(0)}, \widehat  u_0 e^{7 s \Phi(0)/6} )}_{H^2(\T_L)\times H^2(\T_L)}.
	\end{multline}
	This allows to define a linear operator $\mathscr{G}$ defined on the set 
	\begin{multline*}
		\Big\{
			(\widehat  \rho_0, \widehat  u_0, \widehat  f_\rho, \widehat  f_u) \in H^2(\T_L)\times H^2(\T_L) \times  L^2(0,T; H^2(\T_L))\times  L^2(0,T; H^1(\T_L))
			\\
			\hbox{ with } 
			\widehat  f_\rho e^{s \Phi} \in L^2(0,T; H^2(\T_L))
			\hbox{ and }
			\widehat  f_u e^{7 s \Phi/6 } \in L^2(0,T; H^1(\T_L))
		\Big\}
	\end{multline*}
	by $\mathscr{G}(\widehat  \rho_0, \widehat  u_0, \widehat  f_\rho, \widehat  f_u) = (\rho, u)$, where $(\rho, u)$ denotes a controlled trajectory solving \eqref{Navier-Stokes-Linear} with initial condition $(\widehat  \rho_0, \widehat  u_0)$, satisfying the control requirement \eqref{ControlReq} and estimates \eqref{Estimate-Control-Linear-1}--\eqref{Estimate-Control-Linear-2}.
\end{theorem}
\begin{proof}
	We first recover observability estimates for \eqref{Navier-Stokes-Adjoint-Mu-Sigma}. We write down
	\begin{multline*}
		\norm{(\sigma e^{- s \Phi}, q e^{- s \Phi})}_{L^2(H^{-2})\times L^2(H^{-2})} + \norm{(\sigma(0) e^{-s \Phi(0)}, q(0) e^{-s \Phi(0)})}_{H^{-2}\times H^{-3}}
		\\
		= 
		\sup_{\substack{ \norm{(f_r e^{s \Phi}, f_y e^{s \Phi})}_{L^2(H^{2})\times L^2(H^{2})} \leq 1\\ \norm{(r_0 e^{s \Phi(0)}, y_0 e^{s \Phi(0)})}_{H^2 \times H^3} \leq 1}} 
		\langle (f_r, f_y), (\sigma, q)  \rangle_{L^2(H^2), L^2(H^{-2})} + \langle (r_0, y_0), (\sigma(0), q(0)) \rangle_{H^2 \times H^3, H^{-2} \times H^{-3}} 
	\end{multline*}
But by construction, if we associate to $(r_0, y_0) \in H^2(\T_L) \times H^3(\T_L)$ and $f_r, f_y$ satisfying $f_r e^{s \Phi}, f_y e^{s \Phi} \in L^2(0,T;H^2(\T_L))$ the controlled trajectory of \eqref{Control-ReducedEq} in Theorem \ref{Thm-Main-Control-Reduced}, then we get:
	\begin{multline*}
		\langle (f_r, f_y), (\sigma, q)  \rangle_{L^2(H^2), L^2(H^{-2})} + \langle (r_0, y_0), (\sigma(0), q(0)) \rangle_{H^2 \times H^3, H^{-2} \times H^{-3}} 
		\\
		= 
		\langle (g_\sigma, \hbox{div} g_z + \frac{\overline{\rho}^2}{\nu} g_\sigma), (r,y) \rangle_{L^2(H^{-2})\times L^2(H^{-4}), L^2(H^2) \times L^2(H^4)}
		\\
		+
		\langle (\sigma, q), \chi_0 (v_r, v_y) \rangle_{L^2(H^{-2}), L^2(H^2)}.
	\end{multline*}
	Using \eqref{Est-reg-fix-point}, we obtain:
	\begin{multline}
		\norm{(\sigma e^{- s \Phi}, q e^{- s \Phi})}_{L^2(H^{-2})\times L^2(H^{-2})} + \norm{(\sigma(0) e^{-s \Phi(0)}, q(0) e^{-s \Phi(0)})}_{H^{-2}\times H^{-3}}
		\\
		\leq
		C \left(\norm{(g_\sigma e^{- 6s \Phi/7}, g_z e^{- 6s \Phi/7})}_{L^2(H^{-2})\times L^2(H^{-3})} 
		+
		\norm{\chi_0(\sigma e^{-6 s\Phi/7} ,q e^{-6 s \Phi/7} )}_{L^2(H^{-2})\times L^2(H^{-2})}\right).
	\end{multline}
	
Then we use the equation of $z$ in \eqref{Navier-Stokes-Adjoint} to recover estimates on $z$, that we rewrite as follows:
	$$
	\ds - \overline{\rho} (\partial_{t} z + \overline{u} \cdot \nabla z )  - \mu \Delta z 
	 =g_z+ \overline{\rho} \left( 1 - \frac{\lambda+ \mu}{\nu}\right) \nabla \sigma + \frac{\lambda + \mu}{\nu} \nabla q, \text{ in } (0,T) \times \T_L.
	$$
Then we use the duality with the following controllability problem for the heat equation: 
	\begin{equation} \label{Heat-control-dual-z}
	\left\{
		\begin{array}{ll}
			\ds \overline{\rho}( \partial_t  y + \overline{u} \cdot \nabla y )- \Delta y = \tilde{f}_{y}  + v_y {\chi}_{0}, \quad & \hbox{ in }(0,T) \times \T_L,
			\\
			y(0,\cdot) = y_0, \quad y(T, \cdot) = 0 \quad & \hbox{ in } \T_L.
		\end{array}
	\right.
	\end{equation}
	Replacing $s$ by $7s/6 $ in Proposition \ref{Prop-Reg-heat} items 1 \& 2, we get that if $y_0 \in H^2(\T_L)$ and $\tilde f_y e^{7 s \Phi/6} \in L^2(0,T;H^1(\T_L))$, then there exists a controlled trajectory $y$ satisfying \eqref{Heat-control-dual-z} with control $v_y$ with 
	$$
		\norm{y e^{s \Phi}}_{L^2(H^3)} + \norm{v_y e^{s \Phi} }_{L^2(H^1)} \leq C \norm{\tilde f_y e^{7s\Phi/6} }_{L^2(H^1)} + C \norm{y_0 e^{7 s \Phi(0)/6}}_{H^2}.
	$$
	 Arguing by duality, we thus obtain
	\begin{align*}
		&\norm{z e^{ -7 s \Phi/6}}_{L^2(H^{-1})} + \norm{ z(0)e^{-7 s \Phi(0)/6}}_{H^{-2}}
		\\
		& \leq 
		C
		\left(
		\norm{\chi_0 z e^{-s \Phi}}_{L^2(H^{-1})}
		+ 
		\norm{(\sigma,q) e^{- s \Phi}}_{L^2(H^{-2})}
		+
		\norm{g_z e^{- s \Phi}}_{L^2(H^{-3})}
		\right)
		\\
		&\leq
		C \left(
		\norm{(g_\sigma e^{- 6s \Phi/7}, g_z e^{- 6s \Phi/7})}_{L^2(H^{-2})\times L^2(H^{-3})} 
		+
		\norm{\chi_0(\sigma e^{-6s\Phi/7} ,q e^{-6 s \Phi/7} )}_{L^2(H^{-2})\times L^2(H^{-2})}
		\right.
		\\
		& \qquad 
		\left.
		+ \norm{\chi_0 z e^{-s \Phi}}_{L^2(H^{-1})}\right).
	\end{align*}
As $\chi = 1$ in $\hbox{Supp}\, \chi_0$ (recall \eqref{Def-Chi0}), we have $\chi_{0} \chi = \chi_{0}$ and and $\chi_{0} \div z =\chi_{0} \div (\chi z)$.
Now using that $\chi_{0}$ is a multiplier on $H^{-2}$ we get
	\begin{align*}
		\norm{ \chi_0 q e^{-6 s \Phi/7} }_{L^2(H^{-2})} 
		& \leq
		C \norm{ \chi_0 \div z e^{-6 s \Phi/7} }_{L^2(H^{-2})} + C \norm{ \chi_0 \sigma e^{-6 s \Phi/7} }_{L^2(H^{-2})}
		\\
		& \leq C \norm{ \chi z e^{-6 s \Phi/7} }_{L^2(H^{-1})} + C \norm{ \chi \sigma e^{-6 s \Phi/7} }_{L^2(H^{-2})},
	\end{align*}
	and combining the above results, we obtain
	\begin{multline}
		\norm{\sigma e^{- s \Phi}}_{L^2(H^{-2})} + \norm{\sigma(0) e^{-s \Phi(0)}}_{H^{-2}}
		+
		\norm{z e^{ -7 s \Phi/6}}_{L^2(H^{-1})} + \norm{ z(0)e^{-7 s \Phi(0)/6}}_{H^{-2}} 
		\\
		\leq
		C \norm{(g_\sigma e^{- 6s \Phi/7}, g_z e^{- 6s \Phi/7})}_{L^2(H^{-2})\times L^2(H^{-3})} 
		 +
		 C \norm{ \chi( \sigma, z) e^{-6 s \Phi/7} }_{L^2(H^{-2})\times L^2(H^{-1})}.
	\end{multline}
Using that $(\sigma, z)$ satisfies Equation~\eqref{Navier-Stokes-Adjoint}, we again argue by duality to deduce that System~\eqref{Navier-Stokes-Linear} is controllable and the estimate \eqref{Estimate-Control-Linear-1} follows immediately.
	\\
	To conclude \eqref{Estimate-Control-Linear-2}, we look at the equations satisfied by $\rho e^{5 s \Phi/6}$ and $u e^{5s \Phi/6}$ and perform regularity estimates on each equation. To estimate the regularity of $\rho e^{5s \Phi/6}$, as $\Phi$ does not satisfy the transport equation anymore ($\Phi$ is independent of the space variable $x$), this induces a small loss in the parameter $s$, which is reflected by the fact that we estimate $\rho e^{5s \Phi/6}$ instead of $\rho e^{6s\Phi/7}$. This is similar for the estimate on the velocity field $u$.
\end{proof}
\section{Proof of Theorem \ref{Thm-Main}}\label{Sec-Proof-Main}
In this section, we fix the parameter $s = s_0$ so that Theorem \ref{Thm-Control-NS-linear} applies.
We introduce the set on which the fixed point argument will take place:
\begin{align*}
	\mathscr{C}_{R} = \{(\rho,u) \hbox{ with }  & \rho \in L^\infty(0,T; H^2(\T_L)) \cap H^1 (0,T; L^2(\T_L)) 
	\\
	\qquad & u \in L^2(0,T; H^3(\T_L))\cap L^\infty(0,T;H^2(\T_L)) \cap H^1(0,T;H^1(\T_L)), 
	\\
	&\norm{(\rho e^{5 s  \Phi/6}, u e^{5 s \Phi/6} )}_{(L^\infty(H^2) \cap H^1 (L^2))\times (L^2(H^3)\cap L^\infty(H^2) \cap H^1(H^1))}  \leq R\}.
\end{align*}
The precise definition of our fixed point map is then given as follows:
\begin{equation}
	\label{Def-FixedPointMap}
	\mathscr{F}(\widehat  \rho, \widehat  u) = \mathscr{G} (\widehat  \rho_0, \widehat  u_0,f_\rho(\widehat  \rho, \widehat  u), f_u(\widehat  \rho, \widehat  u)), 
\end{equation}
where $\mathscr{G}$ is defined in Theorem \ref{Thm-Control-NS-linear}, $(\widehat  \rho_0, \widehat  u_0)$ is defined in \eqref{Dico-Tilde-rho-u-0-rho-u}  and  $f_\rho (\widehat  \rho,\widehat  u), f_u (\widehat  \rho,\widehat  u)$ are defined in \eqref{SourceTermRho-hat}--\eqref{SourceTermU-hat}. Therefore, our first goal is to check that $\mathscr{F}$ is well-defined on $\mathscr{C}_R$, and for that purpose, we shall in particular show that, for $(\widehat  \rho, \widehat  u) \in \mathscr{C}_R$ one has $\widehat  \rho_0 e^{s \Phi(0)} \in H^2(\T_L)$, $\widehat  u_0 e^{7 s \Phi(0)/6} \in H^2(\T_L)$, $ f_\rho (\widehat  \rho,\widehat  u) e^{s \Phi} \in L^2(0,T; H^2(\T_L))$ and $ f_u (\widehat  \rho,\widehat  u) e^{7s \Phi/6} \in L^2(0,T; H^1(\T_L))$.
\subsection{The map $\mathscr{F}$ in (\ref{Def-FixedPointMap}) is well-defined on $\mathscr{C}_R$}\label{Subsec-WellDefinedMap}
In order to show that $\mathscr{F}$ is well-defined, we first study the maps $\widehat Y$ and $\widehat Z$ defined in \eqref{Def-Hat-Y}--\eqref{Def-Hat-Z} and prove some of their properties, in particular that they are close to the identity map. We can then define the source term $ f_\rho (\widehat  \rho,\widehat  u),  f_u (\widehat  \rho,\widehat  u) $ and the initial data $(\widehat \rho_0, \widehat u_0)$. Accordingly, we will deduce that the map $\mathscr{F}$ in (\ref{Def-FixedPointMap}) is well-defined on $\mathscr{C}_R$ for $R>0$ small enough.
\subsubsection{Estimates on $\widehat Y$ and $\widehat Z$ in (\ref{Def-Hat-Y})--(\ref{Def-Hat-Z})}
We start with the following result:
\begin{proposition}
	\label{Prop-Hat-Y}
	Let $\widehat u \in L^2(0,T; H^3(\T_L)) \cap H^1(0,T; H^1(\T_L))$ with 
	\begin{equation}
		\label{Estimate-Hat-u}
		\norm{\widehat u e^{5 s \Phi/6}}_{  L^2(0,T; H^3(\T_L)) \cap H^1(0,T; H^1(\T_L))}  \leq R.
	\end{equation}
	Then the map $\widehat Y$ defined in \eqref{Def-Hat-Y} satisfies, for some constant $C$ independent of $R>0$:
	\begin{equation}
		\label{Est-I-hat-Y}
		\norm{ (\widehat Y(t,x)- x) e^{ 5s \Phi(t)/12}}_{C^0([0,T]; H^3(\T_L))} + 
		\norm{ (\widehat Y(t,x)- x) e^{ s \Phi(t)/3}}_{C^{1}([0,T]; H^2(\T_L))} \leq C R. 
	\end{equation}
	Therefore, there exists $R_0 \in (0,1)$ such that for all $R \in (0, R_0)$ the map $\widehat Z$ defined in \eqref{Def-Hat-Z} is well-defined and satisfies 
	\begin{align}
		\label{Est-I-hat-Z-1}
		& \norm{ (D\widehat Z(t,\widehat Y(t,x)) - I) e^{ 5s \Phi(t)/12}}_{C^0([0,T]; H^2(\T_L))}
		\leq C R, 
		\\
		\label{Est-I-hat-Z-2}
		& \norm{ (D\widehat Z(t,\widehat Y(t,x)) - I) e^{ s \Phi(t)/3}}_{W^{1/4,5}(0,T;H^{7/4}(\T_L))} \leq C R,
		\\
		\label{Est-D-2-hat-Z}
		&
		\norm{ D^2 \widehat Z(t,\widehat Y(t,x)) e^{5 s \Phi(t)/12} }_{L^\infty(0,T; H^1(\T_L))} \leq C R,
	\end{align}
	and 
	\begin{equation}
		\label{ConditionSupport}
		\chi( \widehat Z(t,x)) = 0 \hbox{ for all } (t,x) \in [0,T] \times \overline\Omega, 
	\end{equation}
	where $\chi$ is defined by \eqref{Def-Chi}. 
\end{proposition}
\begin{proof}
	Let us consider the equation satisfied by the map 
	$$
		\widehat\delta(t,x) = \widehat Y(t,x) -x.
	$$ 
	Using \eqref{Def-Hat-Y}, direct computations show that $\widehat\delta$ satisfies
	\begin{equation}
		\label{Eq-Delta}
		\partial_t \widehat\delta + \overline u \cdot \nabla \widehat\delta = \widehat u, \hbox{ in } (0,T) \times \T_L, \quad \widehat\delta(T,x) = 0, \hbox{ in } \T_L.
	\end{equation}
	Therefore, we immediately get that 
	$$
		\widehat\delta(t,x) = - \int_t^T \widehat u(\tau, x + (\tau - t) \overline u) \, d\tau.
	$$
	As for all $t \in [0,T]$, 
	$$
		\Phi(t) \leq 2 \min_{\tau \in [t,T]} \{ \Phi(\tau)\},
	$$
	we deduce from the above formula and \eqref{Estimate-Hat-u} that
	$$
		\norm{ \widehat\delta e^{ 5s \Phi/12}}_{C^0(H^3)} \leq C R. 
	$$
	Using Equation \eqref{Eq-Delta} and the bound \eqref{Estimate-Hat-u}, we also derive
	$$
		\norm{ \partial_t \widehat\delta e^{ 5s \Phi/12}}_{C^{0}(H^2)} \leq C R, 
	$$
	from which, together with $|s\partial_t \Phi| \leq C e^{s \Phi/12}$ independently of $s$, we immediately deduce \eqref{Est-I-hat-Y}.
	\par
	In particular, 
	$$
		\norm{D \widehat\delta}_{C^0(H^2)} \leq CR, 
	$$
	so that for $R$ small enough, for all $(t,x) \in [0,T] \times \T_L$, $D\widehat Y(t,x) = I + D \widehat\delta(t,x)$ is invertible. Consequently, $\widehat Z$ defined by \eqref{Def-Hat-Z} is well-defined by the inverse function theorem (note that $C^0(H^2) \subset C^0(C^0)$ in dimension $d \leq 3$), and $\widehat Z \in C^0(C^1)$ with 
	$$
		\norm{\widehat Z}_{C^0(C^1)} \leq C R.
	$$ 
	In order to get estimates on $\widehat Z$ in weighted norms, we start from the formula $\widehat Z(t,\widehat Y(t,x)) = x$ and differentiate it with respect to $x$: We obtain, for all $t \in [0,T]$ and $x \in \T_L$, 
	\begin{equation}
		\label{Identity-Jacobians}
		D\widehat Z(t,\widehat Y(t,x)) D\widehat Y(t,x) = I, 
	\end{equation}
	i.e.
	$$
		D \widehat Z(t,\widehat Y(t,x)) (I + D \widehat \delta(t,x)) = I. 
	$$
	Therefore, we can write
	$$
		D \widehat Z(t,\widehat Y(t,x)) = I + \sum_{n= 1}^\infty (-1)^n (D \widehat \delta(t,x))^n. 
	$$
	Using then that $C^0([0,T]; H^2(\T_L))$ is an algebra (as $d \leq 3$) and that $e^{5 s \Phi/12} \geq 1$, 
	\begin{multline*}
		\norm{ (D \widehat Z(t, \widehat Y(t,x))  - I) e^{5s \Phi/12}}_{C^0(H^2)}
		\leq
		\sum_{n = 1}^\infty C^{n-1} \norm{D \widehat \delta e^{5s\Phi/12}}_{C^0(H^2)}^n
		\\
		\leq
		\sum_{n=1}^\infty C^{n-1} (CR)^n \leq \frac{CR}{1 - C^2 R} \leq 2 CR, 
	\end{multline*}
	for $R$ small enough ($R \leq 1/2 C^2$), i.e. \eqref{Est-I-hat-Z-1}. 
	\par
	By interpolation, we have 
	$$
		\norm{D \widehat \delta e^{s \Phi/3}}_{W^{1/4, 5}(H^{7/4})} \leq 
		C
		\norm{D \widehat \delta e^{s \Phi/3}}_{C^1(H^{1})}^{1/4}
		\norm{D \widehat \delta e^{s \Phi/3}}_{C^0(H^{2})}^{3/4}
		\leq CR.
	$$
	We now remark that $W^{1/4,5}(0,T; H^{7/4}(\T_L))$ is an algebra ($5 \times 1/4 > 1$ and $7/4 > d/2$), and we can therefore derive, similarly as above, that
	\begin{multline*}
		\norm{ (D \widehat Z(t, \widehat Y(t,x))  - I) e^{s \Phi/3}}_{W^{1/4, 5}(H^{7/4})}
		\leq
		\sum_{n = 1}^\infty C^{n-1} \norm{D \widehat \delta e^{s\Phi/3}}_{W^{1/4, 5}(H^{7/4})}^n 
		\norm{e^{-s\Phi/3}}_{W^{1/4, 5}(H^{7/4})}^{n-1} 
		\\
		\leq
		\sum_{n=1}^\infty C^{n-1} (CR)^n \leq \frac{CR}{1 - C^2 R} \leq 2 CR, 
	\end{multline*}
	which concludes the proof of \eqref{Est-I-hat-Z-2}.
	\par
	The proof of \eqref{Est-D-2-hat-Z} consists in writing
	$$
		D^2 \widehat Z(t, \widehat Y(t,x)) = D (D \widehat Z(t,\widehat Y(t,x))) (D\widehat Y(t,x))^{-1} 
		= 
		D (D \widehat Z(t,\widehat Y(t,x))) D \widehat Z(t,\widehat Y(t,x)), 	
	$$
	where the last identity comes from \eqref{Identity-Jacobians}. From \eqref{Est-I-hat-Z-1}, we have $D (D \widehat Z(t,\widehat Y(t,x))) e^{5 s \Phi/12} \in L^\infty(H^1)$ and $D \widehat Z(t,\widehat Y(t,x)) e^{5 s \Phi/12} \in L^\infty(H^2)$, hence \eqref{Est-D-2-hat-Z} easily follows as the product of a function in $H^1(\T_L)$ by a function in $H^2(\T_L)$ belongs to $H^1(\T_L)$ ($d \leq 3$).
	\par
	We finally focus on the proof of \eqref{ConditionSupport}. As $\widehat Y(t,x) - x$ belongs to $L^\infty((0,T) \times \T_L)$ by \eqref{Est-I-hat-Y} and $\widehat Y(t,\cdot)$ is a $C^1$ diffeomorphism for all $t \in [0,T]$, 
	$$
		\norm{\widehat Z(t,x) - x}_{L^\infty(L^\infty)} 
		=
		\norm{ \widehat Z(t,\widehat Y(t,x)) - \widehat Y(t,x)}_{L^\infty(L^\infty)} 
		= 
		\norm{x - \widehat Y(t,x)}_{L^\infty(L^\infty)} \leq C R. 
	$$
	In particular, taking $R$ small enough ($R < \varepsilon/C$), condition \eqref{ConditionSupport} is obviously satisfied for $\chi$ defined in \eqref{Def-Chi}.
\end{proof}
\subsubsection{Estimates on $f_\rho(\widehat  \rho, \widehat  u), f_u(\widehat  \rho, \widehat  u)$}
\begin{lemma}
	\label{Lemma-Est-hat-f}
	Let $(\widehat  \rho,\widehat  u) \in \mathscr{C}_{R}$ for some $s \geq s_0$ and $R \in (0,R_0)$, where $R_0$ is given by Proposition \ref{Prop-Hat-Y}. Then we have the following estimate
	\begin{equation}
		\label{Estimate-Fs-hat}
		\norm{(f_\rho (\widehat  \rho,\widehat  u) e^{s \Phi}, f_u (\widehat  \rho,\widehat  u)e^{7 s \Phi/6})}_{L^2(0,T;H^2(\T_L)) \times L^2(0,T;H^1(\T_L))} 
		\leq
		C R^2,
	\end{equation}
	where $f_\rho (\widehat  \rho,\widehat  u), f_u (\widehat  \rho,\widehat  u)$ are defined in \eqref{SourceTermRho-hat}--\eqref{SourceTermU-hat}.
\end{lemma}
\begin{proof}
	In the following, we will repeatedly use the estimates derived in Proposition \ref{Prop-Hat-Y} and the crucial remark that $5/6 + 5/12 > 7/6$.
	\par
\begin{itemize}
\item {\bf Concerning $f_\rho(\widehat  \rho, \widehat  u)$.} We perform the following estimates: as $H^2(\T_L)$ is an algebra in dimension $d \leq 3$, 
	\begin{multline*}
		\norm{(\widehat  \rho D\widehat Z^t(t,\widehat Y(t,x)) :D \widehat  u) e^{s \Phi}}_{L^2(H^2)} 
		\\
		\leq 
		C \norm{\widehat  \rho e^{5s \Phi/6}}_{L^\infty(H^2)} \norm{D\widehat  Z(t,\widehat Y(t,x))}_{L^\infty(H^2)} \norm{\widehat  u e^{5 s  \Phi/6}}_{L^2(H^3)} \leq C R^2,
	\end{multline*}
	and
	\begin{multline*}
		\norm{\overline \rho ( D \widehat Z^t(t,\widehat Y(t,x)) - I) :D\widehat u e^{s \Phi}}_{L^2(H^2)} 
		\\
		\leq
		C \norm{ ( D \widehat Z(t,\widehat Y(t,x)) - I)e^{5s \Phi/12}}_{L^\infty(H^2)} \norm{ D \widehat u e^{5 s \Phi/6}}_{L^2(H^2)} \leq C R^2. 
	\end{multline*}
\item {\bf Concerning $f_u(\widehat  \rho, \widehat  u)$.} Using that the product is continuous from $H^2(\T_L) \times H^1(\T_L)$ into $H^1(\T_L)$,
	$$
		 \norm{\widehat  \rho \partial_t \widehat  u e^{ 7s  \Phi/6} }_{L^2(H^1)}
		\leq 
		 \norm{\widehat  \rho e^{5 s  \Phi/6} }_{L^\infty(H^2)} \norm{\partial_t \widehat  u e^{5 s  \Phi/6}}_{L^2(H^1)}
		\leq C R^2.
	$$
	Using again  that $H^2(\T_L)$ is an algebra, 
	$$
		\norm{\widehat  \rho \overline u \cdot \nabla \widehat  u  e^{ 7s  \Phi/6} }_{L^2(H^1)}
		\leq  
		C \norm{\widehat  \rho e^{5 s  \Phi/6}}_{L^\infty(H^2)}\norm{\widehat  ue^{5 s  \Phi/6}}_{L^2(H^3)} 
		\leq
		C R^2.
	$$
	For $i,j,k, \ell \in \{1, \cdots, d \}$, we get
	\begin{multline*}
		\norm{\partial_{k,\ell} \widehat u_j (\partial_j \widehat Z_{k} (t,\widehat Y(t,x)) - \delta_{j,k}) (\partial_i \widehat Z_{\ell}(t,\widehat Y(t,x)) - \delta_{i, \ell}) e^{7s \Phi/6}}_{L^2(H^1)}
		\\
		\leq
		C \norm{D^2 \widehat u e^{5s \Phi/6}}_{L^2(H^1)} \norm{(D\widehat Z(t,Y(t,x)) - I) e^{5s \Phi/12} }_{L^\infty(H^2)}^2
		\leq
		C R^3.
	\end{multline*}
	Similarly, for $i,j,k \in \{1, \cdots, d \}$, 
	$$
		\norm{ \partial_{i,j} \widehat Z_k(t, \widehat Y(t,x)) \partial_k \widehat u_j e^{7s \Phi/6}}_{L^2(H^1)}
		\leq
		C \norm{D^2\widehat Z(t,Y(t,x)) e^{5s \Phi/12} }_{L^\infty(H^1)} \norm{D u e^{5s \Phi/6}}_{L^2(H^2)}
		\leq 
		CR^2.
	$$
	In order to estimate the terms coming from the pressure, we write $p(\overline{\rho} + \widehat  \rho) - p(\overline\rho) - p'(\overline \rho) \widehat  \rho = \widehat  \rho^2 h(\widehat  \rho)$ where $h$ is a $C^1$ function depending on the pressure law (here we use that the pressure law $p$ belongs to $C^3$ locally around $\overline\rho$), so we have 
	$$
		\nabla (p(\overline{\rho} + \widehat \rho) - p'(\overline{\rho}) \widehat  \rho) = \big( 2 \widehat  \rho h(\widehat  \rho)  + {\widehat{\rho}}^{2} h'(\widehat{\rho}) \big) \nabla \widehat \rho.
	$$
	As $\norm{\widehat  \rho}_{L^\infty(L^\infty)} \leq CR \leq C$, we thus obtain, for $i,j \in \{1, \cdots, d\}$,
	\begin{multline*}
		\norm{\partial_i \widehat Z_j (t, \widehat Y(t,x)) \partial_j (p (\overline \rho + \widehat \rho) - p'(\overline \rho) \widehat \rho) e^{7s \Phi/6}}_{L^2(H^1)}
		\\
		\leq
		C \norm{D \widehat Z(t,\widehat Y)}_{L^\infty(H^2)} \norm{ \widehat \rho e^{5s\Phi/6}}_{L^2(H^2)}  \norm{\nabla \widehat \rho e^{5s\Phi/6}}_{L^2(H^1)}	
		\leq 
		C R^2.
	\end{multline*}
	For $i,j \in \{1, \cdots, d\}$,
	\begin{multline*}
		\norm{ p'(\overline \rho) (\partial_i \widehat Z_{j}(t,\widehat Y(t,x) - \delta_{i,j} ) \partial_j \widehat \rho e^{7 s \Phi/6}}_{L^2(H^1)}
		\\
		\leq
		 C \norm{(D \widehat Z(t,\widehat Y(t,x)) - I)e^{5s \Phi/12}}_{L^\infty(H^2)} \norm{\widehat  \rho e^{5 s \Phi/6}}_{L^2(H^2)}
		 \leq C R^2.
	\end{multline*}
	Combining all the above estimates yields Lemma \ref{Lemma-Est-hat-f}.
\end{itemize}
\end{proof}
\subsubsection{Estimates on $(\widehat\rho_0, \widehat u_0)$}
We finally state the following estimates on $(\widehat \rho_0, \widehat u_0)$ defined in \eqref{Dico-Tilde-rho-u-0-rho-u}:
\begin{lemma}
	\label{Lem-Estimates-Init}
	Let $\widehat u \in L^2(0,T; H^3(\T_L)) \cap H^1(0,T; H^1(\T_L))$ satisfying \eqref{Estimate-Hat-u} for some $R \leq R_0$ given by Proposition \ref{Prop-Hat-Y}. Let $\delta \in (0,1)$ and $(\check \rho_0, \check u_0) \in H^2(\T_L) \times H^2(\T_L)$ with 
	\begin{equation}
		\label{init-cond-torus}
		\norm{(\check \rho_0, \check u_0)}_{H^2(\T_L)\times H^2(\T_L)} \leq C_L \delta.
	\end{equation}
	Define $(\widehat \rho_0, \widehat u_0)$ as in \eqref{Dico-Tilde-rho-u-0-rho-u}. Then there exists a constant $C >0$ independent of $R$ such that 
	\begin{equation}
		\norm{(\widehat \rho_0, \widehat u_0)}_{H^2(\T_L)\times H^2(\T_L)} \leq C \delta.
	\end{equation}
\end{lemma}
\begin{proof}
	It is a straightforward consequence of the estimate \eqref{Est-I-hat-Y} derived in Proposition \ref{Prop-Hat-Y}.	
\end{proof} 
\subsubsection{Conclusion}
Putting together the estimates obtained in Proposition \ref{Prop-Hat-Y}, Lemmas \ref{Lemma-Est-hat-f} and \ref{Lem-Estimates-Init} and using Theorem \ref{Thm-Control-NS-linear}, we get the following result:
\begin{proposition}
	\label{Prop-F-Well-Defined}
	Let $(\check \rho_0, \check u_0)$ in $H^2(\T_L) \times H^2(\T_L)$ satisfying \eqref{init-cond-torus} for some $\delta >0$, $(\widehat \rho, \widehat u) \in \mathscr{C}_R$ for some $R \in (0,R_0)$ with $R_0$ given by Proposition \ref{Prop-Hat-Y}.  
	Then the map $\mathscr{F}$ in \eqref{Def-FixedPointMap} is well-defined, and there exist a constant $C$ such that $(\rho, u) = \mathscr{F}(\widehat \rho, \widehat u)$ satisfies
	\begin{equation}
		\label{Norm-rho-u}
		\norm{( \rho e^{5s \Phi/6}, u e^{5 s \Phi/6})}_{(L^\infty(H^2) \cap H^1 (L^2))\times (L^2(H^3)\cap L^\infty(H^2) \cap H^1(H^1))}
		\leq 
		C R^2 + C \delta. 
	\end{equation}
	Besides, the condition \eqref{ConditionSupport} is satisfied for $\chi$ defined in \eqref{Def-Chi}.
\end{proposition}
Proposition \ref{Prop-F-Well-Defined} is the core of the fixed point argument developed below. 
\subsection{The fixed point argument}
Let $(\rho_0, u_0) \in H^2(\Omega) \times H^2(\Omega)$ satisfying \eqref{Smallness-Init} with $\delta>0$. Choosing an extension $(\check \rho_0, \check u_0)$ of $(\rho_0, u_0)$ satisfying \eqref{Norm-Init-check}, $(\check \rho_0, \check u_0)$ satisfies \eqref{init-cond-torus}. Therefore, from Proposition \ref{Prop-F-Well-Defined}, the map $\mathscr{F}$ in \eqref{Def-FixedPointMap} is well-defined for $(\widehat \rho, \widehat u) \in \mathscr{C}_R$ for $R \in (0,R_0)$, with $R_0 >0$ given by Proposition \ref{Prop-Hat-Y}, and $(\rho, u) = \mathscr{F}(\widehat \rho, \widehat u)$ satisfies \eqref{Norm-rho-u} with some constant $C>0$. We now choose $R \in( 0,R_0)$  such that $CR <1/2$ and $\delta = R/(2C)$, so that as a consequence of Proposition \ref{Prop-F-Well-Defined}, $\mathscr{F}$ maps $\mathscr{C}_R$ into itself. We are then in a suitable position to use a Schauder fixed point argument, and we now fix the parameters $R>0$, $\delta >0$ such that $\mathscr{C}_{R}$ is stable by the map $\mathscr{F}$.
\par
Let us then notice that the set $\mathscr{C}_{R}$ is convex and compact when endowed with the $(L^2(0,T;L^2(\T_L)))^2$-topology, as a simple consequence of Aubin-Lions' Lemma, see e.g. \cite{Simon}.
\par
We then focus on the continuity of the map $\mathscr{F}$ on $\mathscr{C}_{R}$ endowed with the $(L^2(0,T; L^2(\T_L)))^2$ topology. Let us consider a sequence $(\widehat  \rho_n, \widehat  u_n)$ in $\mathscr{C}_{R}$ converging strongly in $(L^2(0,T; L^2(\T_L)))^2$ to some element $(\widehat \rho, \widehat  u)$. The set $\mathscr{C}_{R}$ is closed under the topology of $(L^2(0,T; L^2(\T_L)))^2$ and therefore $(\widehat  \rho, \widehat  u) \in \mathscr{C}_{R}$ and we have the weak-$*$ convergence of $(\widehat  \rho_n, \widehat  u_n)$ towards $(\widehat  \rho, \widehat  u)$ in $(L^\infty(H^2) \cap H^1(L^2)) \times (L^2(H^3) \cap H^1(H^1))$. Using the bounds defining $\mathscr{C}_{R}$, by Aubin-Lions' lemma we also have the following convergences to $(\widehat  \rho, \widehat  u)$:
\begin{equation}
	\label{Convergences-rho-u}
	\begin{array}{ll}
		\widehat  \rho_n \underset{n\to\infty} \rightarrow \widehat  \rho & \hbox{strongly in } L^\infty(0,T; L^\infty(\T_L)),
		\\
		\widehat  \rho_n \underset{n\to\infty}\rightarrow \widehat  \rho & \hbox{strongly in } L^2(0,T; H^1(\T_L)), 
		\\
		\widehat  u_n \underset{n\to \infty} \rightarrow \widehat  u & \hbox{strongly in } L^2(0,T; H^2(\T_L)).
	\end{array}
\end{equation}
To each $\widehat u_n$, we associate the corresponding flow $\widehat Y_n$ defined by 
\begin{equation}
	\label{Def-Y-n}
	\partial_t \widehat Y_n + \overline u \cdot \nabla \widehat Y_n = \overline u + \widehat u_n \hbox{ in } (0,T) \times \T_L, \qquad \widehat Y_n(T,x) = x \hbox{ in } \T_L,
\end{equation}
and the respective inverse $\widehat Z_n$ as in \eqref{Def-Hat-Z} (which is well-defined thanks to Proposition \ref{Prop-Hat-Y}). The sequence $\widehat Y_n$ is bounded in $L^\infty(H^3) \cap W^{1, \infty} (H^2)$ according to \eqref{Est-I-hat-Y} and then converge weakly-$*$ in $L^\infty(H^3) \cap W^{1,\infty}(H^2)$. Passing to the limit in Equation \eqref{Def-Y-n}, we easily get that the weak limit of the sequence $\widehat Y_n$ is $\widehat Y$ defined by \eqref{Def-Hat-Y}. Besides, by Aubin-Lions' lemma and the weak convergence of $\widehat Y_n$ towards $\widehat Y$, we also have the strong convergence of $\widehat Y_n$ to $\widehat Y$ in $W^{1/4, 5}(0,T; H^{11/4}(\T_L))$ and therefore in $C^0([0,T]; C^1(\T_L))$. Consequently, the sequence $\widehat Z_n$ strongly converges to $\widehat Z$ in $C^0([0,T]; C^1(\T_L))$. These strong convergences allow to show that 
\begin{equation}
	\label{Weak-Convergences-flot-init}
	\begin{array}{ccc l}
	D\widehat Z_n(t, \widehat Y_n(t,x)) &\underset{n\to\infty} \rightharpoonup & D \widehat Z(t,\widehat Y(t,x))
	& \hbox{ in } \mathscr{D}'((0,T) \times \T_L),
	\smallskip
	\\
	(\widehat \rho_{0,n}, \widehat u_{0,n}) & \underset{n\to \infty}\rightharpoonup & (\widehat \rho_{0}, \widehat u_{0}) 
	& \hbox{ in } (\mathscr{D}'((0,T) \times \T_L))^2, 
	\end{array}
\end{equation}
where $(\widehat \rho_{0,n}, \widehat u_{0,n}) = (\check \rho_0(\widehat Y_n(0,x)), \check u_0(\widehat Y_n(0,x)))$. 
From the uniform bounds \eqref{Est-I-hat-Z-1}--\eqref{Est-I-hat-Z-2} on the quantity $D\widehat Z_n(t, \widehat Y_n(t,x)) -I$ and Aubin-Lions' Lemma, we also deduce that 
\begin{equation}
	\label{Strong-Convergence-DZ-n}
	D\widehat Z_n(t, \widehat Y_n(t,x)) -I  \underset{n\to \infty}\rightarrow  D\widehat Z(t, \widehat Y(t,x)) -I \hbox{ strongly in } L^2(0,T; H^1(\T_L)).
\end{equation}
Using then the uniform bound \eqref{Est-D-2-hat-Z}, the identity 
$$
	D^2 \widehat Z_n(t,\widehat Y_n(t,x)) = D( D \widehat Z_n(t,\widehat Y_n(t,x))) D\widehat Z_n(t,\widehat Y_n(t,x)),$$
and the convergence \eqref{Strong-Convergence-DZ-n}, we also conclude that 
\begin{equation}
	\label{Weak-Convergence-D2Z-n}
	D^2 \widehat Z_n(t, \widehat Y_n(t,x)) \underset{n\to\infty}\rightharpoonup D^2 \widehat Z(t, \widehat Y(t,x)) \hbox{ weakly in } L^2(0,T; H^1(\T_L)).
\end{equation}
Combining the above convergences, we easily obtain that the functions $f_\rho(\widehat  \rho_n, \widehat  u_n)$ and $f_u(\widehat  \rho_n, \widehat  u_n)$ weakly converge to $f_\rho(\widehat  \rho, \widehat  u)$ and $f_u(\widehat  \rho, \widehat  u)$ in $L^1(0,T; H^1(\T_L))$, and with Lemma~\ref{Lemma-Est-hat-f}, weakly in the weighted Sobolev space described by \eqref{Estimate-Fs-hat}. 
As the control process $\mathscr{G}$ in Theorem \ref{Thm-Control-NS-linear} is linear continuous in $(\widehat \rho_0, \widehat u_0, \widehat  f_\rho, \widehat  f_u)$, it is weakly continuous. Hence $(\rho_n, u_n) = \mathscr{F} (\widehat  \rho_n, \widehat  u_n) = \mathscr{G}(\widehat \rho_{0,n}, \widehat u_{0,n}, f_\rho(\widehat  \rho_n, \widehat  u_n), f_u(\widehat  \rho_n, \widehat  u_n))$ weakly converges to $(\rho, u) = \mathscr{F}(\widehat  \rho, \widehat  u) =  \mathscr{G}(\widehat \rho_{0}, \widehat u_{0}, f_\rho(\widehat  \rho, \widehat  u), f_u(\widehat  \rho, \widehat  u))$ in the sense of distributions. 
But we know that $\mathscr{C}_{R}$ is stable by the map $\mathscr{F}$, so that $(\rho_n, u_n)$ all belong to the set $\mathscr{C}_{R}$, which is compact for the $(L^2(0,T;L^2(\T_L)))^2$ topology. Therefore, $(\rho_n, u_n) = \mathscr{F} (\widehat  \rho_n, \widehat  u_n)$ strongly converges to $(\rho, u) = \mathscr{F}(\widehat  \rho, \widehat  u)$ in $(L^2(0,T; L^2(\T_L)))^2$. 
\par
We conclude by applying Schauder's fixed point theorem to the map $\mathscr{F}$ on $\mathscr{C}_R$. This yields a fixed point $(\rho, u) = \mathscr{F}(\rho,u)$ which by construction solves the control problem \eqref{Navier-Stokes-Tothefixpoint}--\eqref{InitData-0-1}--\eqref{ControlReq}. 
\par
To go back to the original system \eqref{Navier-Stokes-Tothefixpoint-0}, we define $Y$ as the solution of 
$$
	\partial_t Y + \overline u\cdot \nabla Y = \overline u + u \hbox{ in } (0,T) \times \T_L, \qquad Y(T, x) = x \hbox{ in } \T_L, 
$$
and $Z = Z(t,x)$ such that for all $t \in [0,T]$, $Z(t,\cdot)$ is the inverse of $Y(t,\cdot)$, which is well-defined according to Proposition \ref{Prop-Hat-Y}. We then simply set, for all $(t,x) \in [0,T] \times \T_L$,
\begin{equation}
	\label{Dico-Back}
	\check \rho(t,x) = \rho(t, Z(t,x)), \quad \check u(t,x) = u(t, Z(t,x)).
\end{equation}
By construction, $(\check \rho, \check u)$ solves \eqref{Navier-Stokes-Tothefixpoint-0}--\eqref{InitData-0} and the controllability requirement \eqref{ControlReq-0} with control functions $(\check v_\rho, \check v_u)$ defined for $(t,x) \in [0,T]\times \T_L$ by
$$
	\check v_\rho(t,x)= \chi(Z(t,x)) v_\rho(t, Z(t,x)), \quad \check v_u(t,x) = \chi(Z(t,x)) v_u(t,Z(t,x)).
$$
These control functions are supported in $[0,T] \times (\T_L \setminus \overline\Omega)$ thanks to \eqref{ConditionSupport}, so that by restriction on $\Omega$, we get a solution $(\rho_\S,u_S) = (\overline \rho, \overline u) + (\check \rho, \check u)$ of \eqref{Navier-Stokes} satisfying \eqref{NS-Init}--\eqref{NS-Goal}.
\par
To get the regularity estimate in \eqref{Regularity-Controlled-Data}, we first show that the fixed point $(\rho,u)$ of $\mathscr{F}$ satisfies 
$$
	(\rho, u) \in C^0([0,T]; H^2(\T_L)) \times (L^2(0,T; H^3(\T_L)) \cap C^0([0,T]; H^2(\T_L))), 
$$
which is a consequence of \eqref{Estimate-Control-Linear-2}. From these regularity results on $(\rho, u)$, \eqref{Dico-Back} and the regularity estimates obtained on $Z$ in Proposition \ref{Prop-Hat-Y}, we deduce that 
$$
	(\check \rho, \check u) \in C^0([0,T]; H^2(\T_L)) \times (L^2(0,T; H^3(T_L)) \cap C^0([0,T]; H^2(\T_L))), 
$$
and, consequently, \eqref{Regularity-Controlled-Data}. 
%
%
%%%%%%%%%%%%%%%%%%%%%%%%
%
\section{Further comments}\label{Sec-Further}
%
%%%%%%%%%%%%%%%%%%%%%%%%
%
\noindent{\bf The case of non-constant trajectories.} Our result only considers the local exact controllability around a constant state. The next question concerns the case of local exact controllability around non-constant target trajectories, similarly as what has been done in the context of non-homogeneous incompressible Navier-Stokes equations in \cite{BEG} and in the context of compressible Navier-Stokes equation in one space dimension in the recent preprint \cite{Ervedoza-Savel}. We expect such results to be true provided the target trajectory is sufficiently smooth and assuming some geometric condition on the flow of the target velocity field $\overline{u}$. Namely, if we denote by $\overline{X}$ the flow corresponding to $\overline{u}$, i.e. given by 
\begin{equation}
	\frac{d}{dt} \overline{X}(t, \tau, x) = \overline{u} (t, \overline{X}(t, \tau, x)), \quad \overline{X}(\tau, \tau, x) = x, 
\end{equation}
it is natural to expect a geometric condition of the form
\begin{equation}
	\label{Geometric-Condition-Gal}
	\forall x \in \overline\Omega,\, \exists t \in (0,T),\, s.t. \quad  \overline{ X} (t, 0, x) \notin \overline{\Omega}, 
\end{equation}
corresponding to the time condition \eqref{Time-Condition} in the case of a constant velocity field. 
But considering the case of a non-constant velocity field would introduce many new terms in the proof and make it considerably more intricate, including for instance the difficulty to control the density when recirculation appears close to the boundary of $\Omega$. This issue needs to be carefully analyzed and discussed.
\smallskip
\\
\noindent{\bf Controllability from a subset of the boundary.} From a practical point of view, it seems more reasonable to control the velocity and the density from some part of the boundary in which the target velocity field $\overline{u}$ enters in the domain. But in this case, one needs to make precise what are the boundary conditions on the velocity field $u_\S$. 

One could think for instance to Dirichlet boundary conditions of the form $u_\S = \overline{u}$ on the outflow boundary $\Gamma_{\rm out}$. But this would mean that one should find a solution $(\rho,u)$ of the control problem \eqref{Navier-Stokes-Tothefixpoint}--\eqref{InitData}--\eqref{ControlReq} with boundary conditions $u = 0$ on the outflow boundary $\Gamma_{\rm out}$. This would introduce a lot of additional technicalities as the dual variable $z$ would also satisfy Dirichlet homogeneous boundary conditions on the outflow boundary  $\Gamma_{\rm out}$. The variable $q$ in \eqref{Def-q} would therefore have non-homogeneous Dirichlet boundary conditions on $(0,T) \times \Gamma_{\rm out}$ and careful estimates should be done to recover estimates on $z$, for instance based on the delicate Carleman estimates proved in \cite{ImanuvilovPuelYam-2009}. 
%
%
%
%%%%%%%%%%%%%%%%%%%%%%
\appendix
\section{Computations of the equations satisfied by $(\rho,u)$}\label{Appendix}

According to \eqref{Dico-rho-u-tilde-rho-u}, we have, for all $(t, x) \in [0,T] \times \T_L$, 
$$
	\rho(t, X_0(t,T,x)) = \check  \rho(t, X_{\check  u}(t,T,x)), 
	\quad
	u(t, X_0(t,x)) = \check  u(t, X_{\check  u}(t,T,x)), 
$$
so that differentiating in $t$, we easily derive 
\begin{align*}
	& 
	(\partial_t \rho + \overline u\cdot \nabla \rho)(t,Z_{\check  u}(t,x))
	= 
	(\partial_t \check  \rho + (\overline u + \check  u) \cdot \nabla \check  \rho)(t,x)
	\\
	&
	(\partial_t u + \overline u\cdot \nabla u)(t,Z_{\check  u}(t,x))
	= 
	(\partial_t \check  u + (\overline u + \check  u) \cdot \nabla \check  u)(t,x).	
\end{align*}
We then write 
$$
	\check  \rho(t,x) = \rho(t, Z_{\check  u}(t,x)), 
	\quad 
	\check  u(t,x) = u(t,Z_{\check  u}(t,x)), 
$$
which allows us to obtain
\begin{align*}
	&
	\div(\check  u)(t,x) = \sum_{i,j = 1}^d \partial_i Z_{j,\check  u}(t,x) \partial_{j} u_i(t,Z_{\check  u}(t,x)),
	\qquad 
	\Big( \div(\check  u)(t,x) = DZ_{\check  u}^t(t,x) : Du (t,Z_{\check  u}(t,x)) \Big)
	\\
	&
	\Delta \check  u_i (t,x) = 
	 \sum_{j, k, \ell = 1}^d \partial_{k,\ell} u_i(t,Z_{\check  u} (t,x)) \partial_j Z_{k,\check  u} (t,x) \partial_j Z_{\ell,\check  u} (t,x) 
	 + \sum_{k =1}^d \partial_k u_i(t, Z_{\check  u}(t,x)) \Delta Z_{k, \check  u}(t,x),
	\\
	&
	\partial_i \div \check  u(t,x) = 
	\sum_{j,k, \ell =1}^d \partial_{k, \ell} u_j(t,Z_{\check  u}(t,x))\partial_j Z_{k,\check  u} (t,x) \partial_i Z_{\ell, \check  u}(t,x)
	+ \sum_{j,k= 1}^d \partial_{i,j} Z_{k,\check u}(t,x) \partial_k u_j(t,Z_{\check  u}(t,x)),
	\\
	& 
	\partial_i \check  \rho(t,x) = \partial_j \rho (t,Z_{\check  u}(t,x)) \partial_i Z_{j,\check  u}(t,x), 
	\qquad
	\Big( \nabla \check  \rho (t,x) = D Z_{\check  u}(t,x)^t \nabla \rho(t, Z_{\check  u}(t,x)) \, \Big).
\end{align*}
Consequently we get for all $(t,x) \in [0,T] \times \T_L$,
\begin{multline*}
	(\partial_t \check  \rho + (\overline u + \check  u) \cdot \nabla \check  \rho + \overline\rho \div \check  u)(t,x)
	= 
	(\partial_t \rho + \overline u\cdot \nabla \rho + \overline{\rho} DZ_{\check  u}^t (t,Y_{\check  u}(t,\cdot)): Du)(t,Z_{\check  u}(t,x))
	\\
	 = 
	(\partial_t \rho + \overline u\cdot \nabla \rho + \overline{\rho} \div u + \overline{\rho} (DZ_{\check  u}^t (t,Y_{\check  u}(t,\cdot)) - I) : Du)(t,Z_{\check  u}(t,x)). 
\end{multline*}
Then we simply check that 
$$
	\check f(\check \rho, \check u)(t,x) = 
	- 
	\left(
	\rho DZ_{\check u}^t (t,Y(t,\cdot))	: Du 
	\right) (t, Z_{\check u}(t,x)).
$$
We shall therefore look for $\rho$ satisfying the equation 
$$
	\partial_t \rho + \overline u\cdot \nabla \rho + \overline{\rho} \div u = \chi v_\rho + f_\rho(\rho, u),		
$$
where $f_\rho(\rho, u)$ is defined by 
$$
	f_\rho(\rho, u) = - \rho DZ_{\check u}^t (t,Y(t,x)): Du -  \overline{\rho} (DZ_{\check  u}^t (t,Y_{\check  u}(t,x)) - I) : Du).
$$
Similarly, for all $(t,x) \in [0,T] \times \T_L$, 
\begin{align*}
	\lefteqn{
	\left(
	\overline{\rho} (\partial_{t} \check  u_i + (\overline{u} + \check  u)\cdot \nabla \check  u_i )  - \mu \Delta \check  u_i - (\lambda + \mu) \partial_i \div (\check  u) 
+ p'(\overline{\rho}) \partial_i \check  \rho 
	\right)(t,Y_{\check u}(t,x)) 
	}
	\\
	&= 
	\left(
	\overline{\rho} (\partial_{t} u_i + \overline{u} \cdot \nabla u_i )  - \mu \Delta u_i - (\lambda + \mu) \partial_i \div( u) 
+ p'(\overline{\rho}) \partial_i \rho 
	\right)(t,x)
	\\
	& 
	- \mu 
	\left(
	 \sum_{j,k, \ell = 1}^d \partial_{k,\ell} u_i \left(\partial_j Z_{k, \check  u}(t, Y_{\check  u}(t,x)) - \delta_{j,k}\right)\left(\partial_j Z_{\ell, \check  u}(t, Y_{\check  u}(t,x)) - \delta_{j,\ell}\right) + \sum_{k = 1}^d \partial_k u_i \Delta Z_{k, \check  u}(t,Y_{\check  u}(t,x)) 
	\right)
	\\
	&
	- (\lambda + \mu) 
	\left( 	
	\sum_{j,k, \ell =1}^d \partial_{k, \ell} u_j (\partial_j Z_{k,\check  u} (t,Y_{\check  u}(t,x)) - \delta_{j,k}) (\partial_i Z_{\ell, \check  u}(t,Y_{\check  u}(t,x)) - \delta_{i, \ell})
	\right) 
	\\
	& -(\lambda +\mu) 
	\left(
		\sum_{j,k= 1}^d \partial_{i,j} Z_{k,\check u} (t, Y_{\check  u}(t,x)) \partial_k u_j
	\right) 
%	\\
%	& 
	+ p'(\overline{\rho}) 
	\left(
	(DZ_{\check u} ^t(t, Y_{\check  u}(t,x)) - I) \nabla \rho
	\right)  , 
\end{align*}
where as before $\delta_{j,k}$ is the Kronecker symbol. We then compute 
$$
	\check f_{u}(\check \rho, \check u)(t,Y_{\check u}(t,x))
	= 
	- \rho (\partial_t u + \overline u \cdot \nabla u) 
	+ 
	D Z_{\check u}(t,Y_{\check u}(t,x))^t \nabla ( p(\overline \rho + \rho) - p'(\overline\rho) \rho).   
$$
We are therefore led to look for $u$ as a solution of the equation 
\begin{equation*}
	\overline{\rho} (\partial_{t} u + \overline{u} \cdot \nabla u)  - \mu \Delta u - (\lambda + \mu) \nabla \div( u) 
+ p'(\overline{\rho}) \nabla \rho = \chi v_u + f_u(\rho, u), 
\end{equation*}
where $f_u(\rho, u)$ is given componentwise by 
\begin{align*}
	& 
	f_{i,u}(\rho, u) 
	= 
	- \rho (\partial_t u_i + \overline u \cdot \nabla u_i) 
	+ 
	\sum_{j = 1}^d \partial_i Z_{j,\check u}(t,Y_{\check u}(t,x)) \partial_j ( p(\overline \rho + \rho) - p'(\overline\rho) \rho)
	\\
	&
	+ \mu 
	\left(
	 \sum_{j,k, \ell = 1}^d \partial_{k,\ell} u_i \left(\partial_j Z_{k, \check  u}(t, Y_{\check  u}(t,x)) - \delta_{j,k}\right)\left(\partial_j Z_{\ell, \check  u}(t, Y_{\check  u}(t,x)) - \delta_{j,\ell}\right) + \sum_{k = 1}^d \partial_k u_i \Delta Z_{k, \check  u}(t,Y_{\check  u}(t,x)) 
	\right)
	\\
	&
	+ (\lambda + \mu) 
	\left( 	
	\sum_{j,k, \ell =1}^d \partial_{k, \ell} u_j (\partial_j Z_{k,\check  u} (t,Y_{\check  u}(t,x)) - \delta_{j,k}) (\partial_i Z_{\ell, \check  u}(t,Y_{\check  u}(t,x)) - \delta_{i, \ell})
	\right) 
	\\
	& 
	+(\lambda +\mu) 
	\left(
		\sum_{j,k= 1}^d \partial_{i,j} Z_{k,\check u}(t, Y_{\check  u}(t,x)) \partial_k u_j
	\right) 
%	\\
%	& 
	- p'(\overline{\rho}) 
	\left(
	 \sum_{j = 1}^d ( \partial_i Z_{j,\check u}(t,Y_{\check u}(t,x)) - \delta_{i,j}) \partial_j \rho
	\right) .
\end{align*}
\bibliographystyle{plain}

\end{document}